\ifx\documentclass\undefined
\documentstyle[12pt]{article}
\else
\documentclass[12pt]{article}
\fi

\usepackage{color}

\newcommand{\beq}{\begin{eqnarray*}}
\newcommand{\eeq}{\end{eqnarray*}}
\makeatletter
\renewcommand{\theequation}{\thesection.\arabic{equation}}
\@addtoreset{equation}{section}
\def\eqnarray{%
\stepcounter{equation}%
\let\@currentlabel=\theequation
\global\@eqnswtrue
\global\@eqcnt\z@
\tabskip\@centering
\let\\=\@eqncr
$$\halign to \displaywidth\bgroup\@eqnsel\hskip\@centering
$\displaystyle\tabskip\z@{##}$&\global\@eqcnt\@ne
\hfil$\displaystyle{{}##{}}$\hfil
&\global\@eqcnt\tw@$\displaystyle\tabskip\z@{##}$\hfil
\tabskip\@centering&\llap{##}\tabskip\z@\cr}
\makeatother
\newtheorem{theorem}{Theorem}[section]
\newtheorem{lemma}[theorem]{Lemma}

\newtheorem{proposition}[theorem]{Proposition}
\newtheorem{remark}{Remark}[section]

\newsavebox{\toy}
\savebox{\toy}{\framebox[0.65em]{\rule{0cm}{1ex}}}
\newcommand{\QED}{\usebox{\toy}}
\def\nlni{\par\ifvmode\removelastskip\fi\vskip\baselineskip\noindent}
\newenvironment{proof}{\nlni\begingroup\it Proof.\rm}{
\endgroup\vskip\baselineskip}

\begin{document}
\setlength{\baselineskip}{15pt}
\title{
A generalization of carries process\\ and riffle shuffles
}
\author{
Fumihiko Nakano
\thanks{
Department of Mathematics,
Gakushuin University,
1-5-1, Mejiro, Toshima-ku, Tokyo, 171-8588, Japan.
e-mail : 
fumihiko@math.gakushuin.ac.jp}
\and 
Taizo Sadahiro
\thanks{Department of Computer Science, 
Tsuda College, 
2-1-1, Tsuda, Kodaira City, 187-8577, 
Tokyo, Japan.
e-mail : sadahiro@tsuda.ac.jp}
}
\maketitle
\begin{abstract}
As a continuation to our previous work
\cite{NS1}, 
we consider 
a generalization of carries process.
Our results are : 
(i) 
right eigenvectors of the transition probability matrix, 
(ii)
correlation of carries between different steps, and 
(iii)
generalized riffle shuffle
whose corresponding descent process has the same distribution as that of the generalized carries process. 
\end{abstract}


\section{Introduction}
\subsection{Background and definition}
Carries process 
is the Markov chain of carries in adding array of numbers. 
It was 
Holte \cite{Holte} 
who first studied the carries process, and he found many beautiful properties, e.g., 
the eigenvalues of the transition probability matrix 
$P$ 
consist of negative powers of the base 
$b$,  
the eigenvectors of 
$P$ are independent of 
$b$, 
and Eulerian numbers appear in the stationary distribution.
Diaconis-Fulman \cite{DF1, DF2, DF3} 
found connections to different subjects, e.g., the carries process has the same distribution as the descent process induced by the repeated applications of the riffle shuffle, 
and 
the array of the left eigenvectors of $P$ 
coincides with the Foulkes character table of 
$S_n$. 

In 
\cite{NS1}, 
we considered a generalization of the carries process in the sense that 
(i)
we take various digit sets, and 
(ii)
we also consider negative base. 
There we obtained
(i) the left eigenvectors of 
$P$, 
and 
(ii) the limit theorem which yields the distribution of the sum of i.i.d. uniformly distributed random variables on 
$[0,1]$. 
This paper 
is the continuation of 
\cite{NS1} ; 
here we study  
(i) the right eigenvectors of 
$P$, 
which yields the correlation of carries of different steps, and 
(ii) the equivalence to the descent process induced by the repeated generalized riffle shuffles on the colored permutation group.  
\cite{NS3}
is a review article of our results obtained so far.

In what follows, 
we first recall the definitions of carries process and some results in 
\cite{NS1}
(subsection 1.2), 
and then state our results in this paper(subsection 1.3). 
To simplify the statements, 
we shall discuss the positive/negative base simultaneously. 
Let 
$\pm b \in {\bf Z}$($b \ge 2$)
be the base and let 
$
{\cal D}_d :=
\{ d, d+1, \cdots, d + b-1 \}
$
be the digit set such that 
$
1-b \le d \le 0
$
to have 
$0 \in {\cal D}_d$. 
Then any 
$x \in {\bf N}$
has the unique representation 
%
\begin{eqnarray*}
x &=& a_N (+ b)^N + a_{N-1} (+ b)^{N-1} + \cdots + a_0, 
\quad
a_k \in {\cal D}_d,
\\
x &=& a_N' (- b)^N + a_{N-1}' (- b)^{N-1} + \cdots + a_0', 
\quad
a'_k \in {\cal D}_d.
\end{eqnarray*}
%
In adding 
$n$
numbers under this representation, let 
$C_{k-1}^{\pm}$
be the carry from the 
$(k-1)$-th digit which belongs to a set 
${\cal C}(\pm b, n)$ 
to be specified in Proposition \ref{carry}. 
In the $k$-th digit, 
we take  
$X_{}, \cdots, X_{n}$
uniformly at random from 
${\cal D}_d$
and then the carry 
$C_k^{\pm} \in {\cal C}(\pm b, n)$ 
to the 
$(k+1)$-th digit 
is determined by the following equation :  
\begin{equation}
C_{k-1}^{\pm} + X_{1} + \cdots + X_{n}
=
C_k^{\pm}(\pm b) + r, 
\quad
r \in {\cal D}_d.
\label{carriesprocess}
\end{equation}
The process
$\{ C_k^{\pm} \}_{k=0}^{\infty}$
is 
Markovian with state space
${\cal C}(\pm b, n)$, 
which we call the 
{\bf $n$-carries process over 
$(\pm b, {\cal D}_d)$}.
Holte's carries process 
corresponds to the case where 
the base is positive and 
$d=0$. 
%
\subsection{Our previous results}
In this subsection 
we recall 
some results in 
\cite{NS1}
related to those in this paper. 
Let 
\beq
l_{\pm}
&=&
l(\pm b, d)
:=
\left\{
\begin{array}{cc}
\frac {d}{b-1} & ((+b)-case) \\
- \frac {b+d}{b+1} & ((-b)-case) \\
\end{array}\right.
\eeq
Then the carry set 
${\cal C}(\pm b, n)$ 
is explicitly given by Proposition \ref{carry} below. 
\begin{proposition}
{\bf \quad}\\
\label{carry}
(1)
The carry set 
${\cal C}(\pm b, n)$
in the 
$n$-carries process over 
$(\pm b, {\cal D}_d)$ 
is given by 
\beq
\quad
&&
{\cal C}(\pm b, n)
=
\{\min {\cal C}_{\pm}, \min {\cal C}_{\pm}+1, \cdots, \max {\cal C}_{\pm}\}
\\
%
&& 
\min {\cal C}_{\pm} :=
\left\lfloor (n-1) l_{\pm} \right\rfloor,
\quad
\max {\cal C}_{\pm} :=
\left\lceil (n-1)(l_{\pm} +1) \right\rceil.
\eeq
(2)
The number of elements of 
${\cal C}(\pm b, n)$ 
is  
\beq
\sharp {\cal C}(\pm b, n)
&=&
\left\{
\begin{array}{cc}
n & ((n-1)l_{\pm} \in {\bf Z}) \\
n+1 & ((n-1)l_{\pm} \notin {\bf Z}). \\
\end{array}\right.
\eeq
\end{proposition}
$\sharp S$
is the number of elements of a finite set $S$.
If 
$(n-1) l_{\pm} \notin {\bf Z}$, 
$\sharp{\cal C}(\pm b, n)$ 
is larger than that of Holte's carries process. 
We 
introduce the following notation which is an important parameter to describe our results. 
\beq
p 
=
p(\pm b, d, n) 
:=
\frac {1}
{1 - \langle (n-1) l_{\pm} \rangle}
=
\left\{
\begin{array}{cc}
1 & ((n-1)l_{\pm} \in {\bf Z}) \\
\langle (n-1)(-l_{\pm}) \rangle^{-1}
& ((n-1)l_{\pm} \notin {\bf Z}) \\
\end{array}
\right.
%
\eeq
where
$\langle x \rangle
:=
x - \lfloor x \rfloor$ 
is the fractional part of 
$x$. 
$\sharp {\cal C}(b, n) = n$
if and only if 
$p=1$,  
including the case of Holte's carries process.
To study further, 
let us change variables : 
$X_j = Y_j + d$, $r = s + d$, 
$C_{k}^{\pm} = \kappa_k^{\pm} + \min {\cal C}_{\pm}$ 
so that 
\beq
&&
Y_j, s \in {\cal D}(b):=\{0, 1, \cdots, b-1\}, 
\\
&&
\kappa_k^{\pm} \in {\cal C}_p(n):=
\left\{
\begin{array}{cc}
\{0, 1, \cdots, n-1 \}
&
(p=1) \\
\{0, 1, \cdots, n\} 
&
(p \ne 1) \\
\end{array}
\right.
\eeq
Substituting them into (\ref{carriesprocess}), 
we have
\begin{eqnarray}
&&
\cases{
\kappa_{k-1}^{+} + Y_1 + \cdots + Y_n + A_{+}(b)
=
\kappa_{k}^{+}b + s
& ($(+b)$-case) \cr
\kappa_{k-1}^{-} + Y_1 + \cdots + Y_n + A_{-}(b)
=
(n-\kappa_{k}^{-})b + s
& ($(-b)$-case) \cr
}
\label{bnp}
\\
&&
\mbox{ where }\quad
A_+(b) := \frac{b-1}{p^*}, 
\quad
A_-(b) := \frac {b+1}{p}-1, 
\quad
\frac 1p + \frac {1}{p^*} = 1. 
\nonumber
\end{eqnarray}
It says that, 
$\{ \kappa_k^{\pm} \}$
is similar to  
$(n+1)$-carries process over 
$(b, {\cal D}_0)$ 
except that 
$X_{n+1} = A_{\pm}(b)$ 
and 
$C_{k}^-$ 
in 
(\ref{carriesprocess}) 
is replaced by 
$n - C_k^-$.
Note that 
$A_{\pm}(b) \in {\bf N}$. 
Conversely, if 
$p \ge 1$ and 
$\frac {b \mp 1}{p} \in {\bf N}$
(equivalently, 
$b \equiv \pm 1 \pmod p$
if 
$p \in {\bf N}$), 
then 
(\ref{bnp}) 
still defines a Markov chain
$\{ \kappa_k^{\pm} \}$ 
on 
${\cal C}_p(n)$, 
even if it does not correspond to a 
$n$-carries process over 
$(\pm b, {\cal D}_d)$
for some 
$d$. 
We call 
$\{ \kappa_k^{\pm} \}$
{\bf $(\pm b, n, p)$-carries process}.
A simple 
examination of (\ref{bnp}) implies that 
$(\pm b, n, p)$-carries process 
is irreducible and aperiodic.
When 
$p=1$(resp. $p=2$), 
$(+b, n, p)$-carries process 
coincides with that of Holte's carries process
(resp. type B process in \cite{DF2}). 
Let 
$P_p^{\pm} := \{ P_p^{\pm}(i,j) \}_{i, j \in {\cal C}_p(n)}$
be the transition probability matrix for the 
$(\pm b, n, p)$-carries process : 
\[
P_p^{\pm}(i,j)
:=
{\bf P}\left(
\kappa_1^{\pm} = j \, | \, \kappa_0^{\pm} = i
\right), 
\quad
i, j \in {\cal C}_p(n).
\]
Then 
$P_p$ 
and its left eigenvectors are given below. 
\begin{proposition}
\label{prob}
%
\begin{eqnarray}
P_p^{\pm}(i, j)
&=&
\frac {1}{b^n}
\sum_{r \ge 0}
(-1)^r
\left( \begin{array}{c}
n+1 \\ r
\end{array} \right)
\left( \begin{array}{c}
n + B_p^{\pm}(i, j) - br \\ n
\end{array} \right)
1(B_p^{\pm}(i, j) - br \ge 0)
\nonumber
\\
%
%
\mbox{ where }
&&
B_p^+(i, j) 
:= 
\left( j + \frac 1p \right) b
- 
\left(
i + \frac 1p
\right)
\nonumber
\\
&&
B_p^-(i, j) 
:= 
\left( -j + 1 - \frac 1p \right) b
- 
\left(
i + \frac 1p
\right)
 + nb
 \nonumber
\\
&&
i, j \in {\cal C}_p(n)
\nonumber
\end{eqnarray}
where 
$1(E)$
is the indicator function of the event 
$E$, 
that is, 
$1(E) = 1$
if 
$E$
is true and 
$1(E) = 0$
otherwise. 
\end{proposition}
\begin{theorem}
\label{left}
The eigenvalues and the left eigenvectors of 
$P_p^{\pm}$
are given by
\beq
P_p^{\pm} &=& L_p^{-1} D_b^{\pm} L_p
\\
\mbox{ where }
\quad
D_b^{\pm} &:=& \mbox{diag}
\left(
1, \left(\pm \frac 1b \right), \cdots, 
\left(\pm \frac 1b \right)^{\sharp {\cal C}_p(n)-1}
\right)
\\
L_p
&:=&
\{ v_{ij}^{(p)}(n) \}_{0 \le i, j \le \sharp{\cal C}_p(n)-1}
\\
v_{ij}^{(p)}(n)
&:=&
\sum_{r=0}^j 
(-1)^r 
\left( \begin{array}{c}
n+1 \\ r
\end{array} \right)
\left\{
p(j-r)+1
\right\}^{n-i}.
\eeq
\end{theorem}
\begin{remark}
\label{descent}
$\{ v_{i,j}^{(p)}(n) \}$
satisfy the following recursion relation.
\beq
v^{(p)}_{i,j}(n)
&=&
(pj+1) v^{(p)}_{i, j}(n-1)
+
\{p(n+1-j)-1\} v^{(p)}_{i, j-1}(n-1)
\eeq
from which we have, if 
$p \in {\bf N}$, 
$v_{0, k}^{(p)}(n)$
is equal to the number of elements in the colored permutation group  
$G_{p, n} (\simeq {\bf Z}_p \wr S_n)$
whose descent is equal to $k$.
\end{remark}
\begin{remark}
\label{Foulkes}
Miller \cite{Miller}
studied the Foulkes character in the general complex reflection groups. 
According to his results, if 
$p \in {\bf N}$, 
$L_p$ 
coincides with the Foulkes character table of 
${\bf Z}_p \wr S_n$. 
\end{remark}
\begin{remark}
\label{dualityL}
The row 
eigenvectors of 
$L_{p^*}$
are the ``reverse" of those of 
$L_p$
in the following sense. 
\beq
v_{i, j}^{(p^*)}(n)
&=&
(-1)^i
\left( 
\frac {p^*}{p}
\right)^{n-i}
v_{i, n-j}^{(p)}(n).
\eeq
\end{remark}
%
\subsection{Results in this paper}
Here 
we state the results obtained in this paper. 
\subsubsection{Right eigenvectors}
%
%
%
%
\begin{theorem}
\label{right}
%
Let 
\[
R_p := L_p^{-1} = \{ u_{ij}^{(p)}(n) \}_{i, j=0, \cdots, \sharp{\cal C}_p(n)-1}
\]
be the matrix composed of the right eigenvectors of 
$P_p^{\pm}$. 
Then its components are given by 
\begin{equation}
u_{ij}^{(p)}(n)
=
\sum_{k=i}^n
\sum_{l=n-j}^k
\frac {
s(k, l) (-1)^{n - j - l}
}
{
k! \;  p^l
}
\left( \begin{array}{c}
l \\ n-j
\end{array} \right)
\left( \begin{array}{c}
n-i \\ n-k
\end{array} \right).
\label{RHS}
\end{equation}
\end{theorem}
$s(n, k)$
is the Stirling number of the first kind : 
\[
s(n, k):=(-1)^{n-k}
\sharp 
\left\{
\sigma \in S_n 
\, | \, 
\sigma 
\mbox{ has $k$ cycles }
\right\}.
\]
\begin{remark}
RHS in (\ref{RHS}) in 
Theorem \ref{right}
appears in 
\cite{Miller}
in a different but related  context(the inverse of the Foulkes character table). 
By Corollary 9.1 in \cite{Miller}
and 
by Lemma \ref{shufflecounting} in Appendix, 
$u_{ij}^{(p)}$ 
has a different expression.
\begin{eqnarray}
u^{(p)}_{ij}(n)
&=&
[x^{n-j}]
\left( \begin{array}{c}
n + \frac {x-1}{p} - i \\ n
\end{array} \right)
\label{different}
\\
&=&
[x^{n-j}]
\sharp \left\{ 
\sigma \in G_{p, n}
\, \Bigl| \, 
\sigma : (x, n, p)\mbox{-shuffle with  }
d (\sigma^{-1}) = i
\right\}
\nonumber
\end{eqnarray}
where 
$[x^n]f(x)$
is the coefficient of 
$x^n$
in the polynomial 
$f(x)$, 
and 
$(b, n, p)$-shuffle
is a generalized riffle shuffle to be defined later. 
\cite{DF4}
derives the right eigenvectors of 
$P_2^+$
for 
$p=2$, 
which is consistent with our result. 
\end{remark}
\begin{remark}
\label{dualityR}
The counterpart 
of duality of 
$L_p$
(Remark \ref{dualityL})
is
\beq
u^{(p^*)}_{ij} (n)
=
(-1)^j
\left(
\frac {p}{p^*} 
\right)^{n-j}
u^{(p)}_{n-i, j}(n).
\eeq
In other words, 
the column vectors of 
$R_p$
and
$R_{p^*}$
are reverse of each other, up to constants.  
\end{remark}
For the combinatorial aspect of 
$R_p$, 
we show 
\begin{theorem}
\label{StirlingFrobenius}
$w_j(n)
:=
n! p^n u_{0, n-j}^{(p)}(n)$
satisfies the following recursion relation. 
\beq
w_j(n) 
&=&
(pn-1) w_j(n-1) + w_{j-1}(n-1) 
\\
w_0(0) &=& 1
\eeq
\end{theorem}
Theorem \ref{StirlingFrobenius} 
implies that 
$w_j (n)$
is equal to the Stirling-Frobenius cycle number of parameter $p$
\cite{SF}.
%
\subsubsection{Correlation of carries}
As is done in \cite{DF3}, 
we can make use of 
Theorem \ref{right} 
to compute the expectations and correlations between carries at different steps.
\begin{theorem}
\label{zerostart}
%
Let 
${\bf E}[ \;\cdot\; | \kappa_0^{\pm} = i]$
be the expectation value conditioned  
$\kappa_0^{\pm} = i$.
Then we have
\beq
&(1)& \quad
{\bf E}[ \kappa_r^{\pm} \, | \, \kappa_0^{\pm} = i ]
=
\frac {1}{(\pm b)^r}
\left( i + \frac 1p - \frac {n+1}{2} \right)
- \frac 1p + \frac {n+1}{2}
\\
&(2)& \quad
\mbox{Var }
(\kappa_r^{\pm} \, | \, \kappa_0^{\pm} = i)
=
\frac {n+1}{12}
\left(
1 - \frac {1}{(\pm b)^{2r}}
\right)
\\
&(3)& \quad
\mbox{Cov }(
\kappa^{\pm}_s, \kappa^{\pm}_{s+r} \, | \, 
\kappa^{\pm}_0 =i)
=
\frac {1}{(\pm b)^r}
\frac {n+1}{12}
\left(
1 - \frac {1}{(\pm b)^{2s}}
\right).
\eeq
\end{theorem}
\begin{theorem}
\label{pistart}
%
Let 
${\bf E}_{\pi}[ \;\cdot\; ]$
be the expectation value conditioned that  
$\kappa_0^{\pm}$ 
obeys the stationary distribution. 
Then we have
\beq
&(1)& \quad
{\bf E}_{\pi}[ \kappa^{\pm}_0 ] 
=
\frac {n+1}{2} - \frac 1p, 
\\
&(2)& \quad
\mbox{Cov}_{\pi}(\kappa^{\pm}_r, \kappa^{\pm}_0)
=
\frac {1}{(\pm b)^r} \cdot \frac {n+1}{12}.
\eeq
\end{theorem}

We note that the variances and covariances do not depend on 
$p$. 
For
$(-b)$-case, 
$\kappa_r^-$
and
$\kappa_0^-$
are negatively correlated when 
$r$
is odd.
Theorems \ref{zerostart}, \ref{pistart} 
are proved in Appendix. 
%
\subsubsection{Relation to a generalized riffle shuffle : $(+b)$-case}
In this subsection 
we assume 
$p \in {\bf N}$ 
and study 
the relation between the 
$(b, n, p)$-carries process and the descent of the (generalized) riffle shuffle. 
First of all, 
we define a ordering on the set 
$\Sigma :=[n]  \times {\bf Z}_p$
as follows
($[n] :=\{ 1, 2, \cdots, n \}$). 
\beq
&&
(1,0) < (2, 0) < \cdots < (n, 0)
\\
<&&
(1, p-1) < (2, p-1) < \cdots < (n, p-1) 
\\
<&&
(1, p-2) < (2, p-2) < \cdots < (n, p-2) 
\\
&&
\cdots 
\\
<&&
(1, 1) <  \cdots < (n, 1).
\eeq
For 
$q \in {\bf Z}_p$
let 
$T_q : \Sigma \to \Sigma$
be a shift given by
$T_q (i, r) := (i, r+q)$, 
$(i, r) \in \Sigma$.
The group 
$G_{p, n} (\simeq {\bf Z}_p \wr S_n)$
of colored permutations 
is the set of bijections 
$\sigma : \Sigma \to \Sigma$ 
such that  
$\sigma \circ T_1  = T_1 \circ \sigma$. 
We set  
$(\sigma(i), \sigma^c(i)) := \sigma(i, 0) \in \Sigma$, 
$i=1, 2, \cdots, n$. 
Then 
$\sigma \in G_{p, n}$
is characterized by 
$\{ ( \sigma(i), \sigma^c (i) ) \}_{i=1}^n$,  
and henceforth we abuse the notation and write 
$\sigma = \left(
(\sigma(1), \sigma^c(1)), (\sigma(2), \sigma^c(2)), \cdots, (\sigma(n), \sigma^c(n))
\right)$.
We say that 
$\sigma
\in 
G_{p, n}$
have a descent at 
$i$ 
if and only if 
$(\sigma(i),\sigma^c(i)) > (\sigma(i+1), \sigma^c(i+1))$
(for $i=1, 2, \cdots, n-1$)
and 
$\sigma^c(n) \ne 0$
(for $i=n$). 
We denote by 
$d(\sigma)$
the number of descents of 
$\sigma$. 

Here we shall consider 
the generalized riffle shuffle.
Roughly speaking, the 
{\bf $(+b, n, p)$-shuffle}
is to carry out the usual 
$b$-shuffle to 
$n$
``cards with $p$ colors", 
but for the
$(jp+r)$-th pile, 
apply  
$T_r$-shift.
To be precise, 
it is defined as follows : 

(i)
Pick up numbers uniformly at random from 
${\cal D}(b) := \{ 0, 1, \cdots, b-1 \}$
$n$ times yielding an array of numbers 
${\bf A} := (a_1, a_2, \cdots, a_n) \in {\cal D}(b)^n$, 
which we call labels, 

(ii)
Sort 
$1,2, \cdots, n \in [n]$
according to the order of the labels 
$a_1, \cdots, a_n$
so that for each 
$i=1, \cdots, n$
we have 
$k_i \in [n]$
corresponding to 
$a_i$
(this is the same procedure to construct the  
$b$-shuffle from its GSR representation), 
and, 

(iii)
If 
$a_i \equiv q_i \in {\bf Z}_p \pmod p$, 
set  
$(\sigma(i), \sigma^c(i)):=(k_i, q_i)$, 
$i=1, \cdots, n$, 
which determines   
$\sigma \in G_{p, n}$. 

\noindent
We denote  
$\sigma$
by 
$\pi_b [{\bf A}] \in G_{p, n}$, 
and call 
${\bf A}=(a_1, \cdots, a_n)$
the GSR representation of 
$\sigma$. 
We show an example of 
$(7,8,3)$-shuffle below.
In this example
$(b, n, p) = (7, 8, 3)$, 
and 
${\bf A} = (4,1,6,3,0,5,0,2)\in {\cal D}(7)^8$.
Then
$(6,3,8,5,1,7,2,4) \in S_8$
is the corresponding element in the $b$-shuffle.
Next, 
$a_1 = 4 \equiv 1 \pmod 3$
so that  
$\sigma^c(1) = 1$, 
and 
$a_2 \equiv 1 \pmod 3$, 
so that 
$\sigma^c (2)=1$,  
and so on.  
We thus obtain  
the pairs of numbers
$(6,1), (3,1), \cdots, (4,2)$ 
in the right column in the table below,  
which  represent the corresponding element 
$\sigma$ in 
$G_{3, 8}$. 
\beq
\begin{array}{cc}
{\bf A} & (\sigma(i), \sigma^c(i)) \\ \hline
4 & (6,1) \\
1 & (3,1) \\
6 & (8,0) \\
3 & (5,0) \\
0 & (1,0) \\
5 & (7,2) \\
0 & (2,0) \\
2 & (4,2) \\ \hline
\end{array}
\eeq
The 
$b$-shuffle
is the special case where 
$p=1$. 
However, 
$(b, n, 2)$-shuffle is slightly different from the 
type B shuffle in \cite{DF2} ; 
In the latter, 
we let the 
$(2j+1)$-th pile upside down
(that is, the order of numbers corresponding to odd 
$a$'s 
are reversed) so that it can be realized as a real card shuffle, 
and that is why our order on 
$[n] \times {\bf Z}_p$ 
(for $p=2$) 
is different from that in 
\cite{DF2}.  
In the example below, we set 
$(b, n, p) = (3, 6, 2)$
and 
${\bf A} = (2, 1, 0, 1, 0, 1)
\in {\cal D}(3)^6$. 
The column in the middle
in the table below  
is the corresponding element 
$\pi_3 [{\bf A}]$ 
in 
$(3, 6, 2)$-shuffle.
The column in the right 
is the corresponding 
type B-shuffle, where 
the 1-st pile 
$(3,4,5)$
is turned down to 
$(5, 4, 3)$. 
\beq
\begin{array}{ccc}
{\bf A} & (3,6,2) & \mbox{type}B \\ \hline
2 & (6,0) & (6,0) \\ 
1 & (3,1) & (5,1) \\
0 & (1,0) & (1,0) \\
1 & (4,1) & (4,1) \\
0 & (2,0) & (2,0) \\
1 & (5,1) & (3,1) \\ \hline
\end{array}
\eeq
Let 
$\{ \sigma_r \}_{r=0}^{\infty}$
be a Markov chain on 
$G_{p, n}$ 
obtained by the repeated applications of 
$(+b, n, p)$-shuffles with  
$\sigma_0 = id$.
For simplicity, 
we shall call 
$\{ \sigma_r \}$
{\bf a sequence induced by  
$(b, n, p)$-shuffle}. 
Using Lemma \ref{shufflecounting}
we can see that 
$\{ \sigma_r \}$
is irreducible and aperiodic. 

Let 
$\{ \kappa_r^+ \}_{r=0}^{\infty}$ 
be the 
$(+b, n, p)$-carries 
process with 
$\kappa_0^+ = 0$.
Then we have the following theorem,  
which is a generalization of the results by Diaconis-Fulman\cite{DF1}.
\begin{theorem}
\label{generalshuffle}
%
Let 
$p \in {\bf N}$, 
$b \equiv 1 \pmod p$. 
Then we have 
$\{ \kappa_r^+ \} \stackrel{d}{=} \{ d(\sigma_r) \}$. 
\end{theorem}
Since 
$\{ \sigma_r \}$
has uniform stationary distribution, 
Theorem \ref{generalshuffle}
gives alternative proof of the fact that the stationary distribution of 
$(b, n, p)$-carries process 
gives the descent statistics of 
$G_{p, n}$ 
for 
$p \in {\bf N}$. 
For the proof 
of Theorem \ref{generalshuffle}, 
we construct a bijection from the summand of 
$(b, n, p)$-carries process to the GSR representations of 
$(b, n, p)$-shuffle. 
In addition to 
the argument in \cite{DF1}, 
we need to introduce  a bijection 
$f_b$ 
on 
${\cal D}(b)$ 
to relate the carries in the summands 
to the descents of elements of 
$G_{p, n}$. 
%
%
\subsubsection{Relation to a generalized riffle shuffle : $(-b)$-case}
We consider a ``shuffle" 
corresponding to the 
carries process with negative base. 
We begin by simpler ones which holds for $p=1,2$ only. 
For given 
$\sigma = (\sigma(1), \cdots, \sigma(n)) \in S_n$, 
let
$R_1 \sigma \in S_n$ 
\[
(R_1 \sigma) (k) := n + 1 - \sigma(k), 
\quad
k = 1, 2, \cdots, n
\]
be its ``reverse".
Let 
\[
S_1^- := R_1 \circ \mbox{($(+b, n, 1)$-shuffle)}
\]
be the operation of carrying out the reverse  after a $(+b, n, 1)$-shuffle. 
Let 
$\{ \tilde{\sigma}_r \}_{r=0}^{\infty}$
($\tilde{\sigma}_r  := (S_1^-)^r \sigma_0$, 
$r = 1, 2, \cdots$)
be the corresponding Markov chain on 
$S_n$ 
with  
$\sigma_0 = id$, 
and let 
$\{ \kappa^{-}_r \}_{r=0}^{\infty}$
be the 
$(-b, n, 1)$-carries process
with  
$\kappa_0^{-} = 0$.
Then we have a counterpart of 
Theorem \ref{generalshuffle}. 
\begin{theorem}
\label{p1}
Let 
$p=1$. 
Then we have 
$\{ \kappa^{-}_r \}
\stackrel{d}{=}
\{ d(\tilde{\sigma}_r) \}$. 
\end{theorem}
Similar result
also holds for 
$p=2$
case where 
$R_1$, $S_1^-$ 
are replaced by 
\[
(R_2 \sigma) (i, 0)
:=
(n+1-\sigma(i), \sigma^c(i)+1), 
\;
S_2^- := R_2 \circ \mbox{($(+b, n, 2)$-shuffle)}.
\]
For the proof of 
Theorem \ref{p1}, 
we construct a connection between 
$(+b, n, p)$-carries process and 
$(-b, n, p)$-carries process, 
and use Theorem \ref{generalshuffle}. 
%

To state the results for general integer  $p$'s, 
we need to introduce another notion of descent on 
$G_{p, n}$. 
We define the ``dash-order" $<'$ on 
$\Sigma$ : 
\beq
&&
(1,0) <' (2,0) <' \cdots <' (n,0) \\
<'&&
(1,1) <' (2,1) <' \cdots <' (n,1) \\
<'&&
\cdots 
\\
<'&&
(1, p-1) <' (2,p-1) <' \cdots <' (n, p-1)
\eeq
we say that 
$\sigma \in G_{p, n}$
has a dash-descent at 
$i$ 
if and only if 
$(\sigma(i), \sigma^c(i)) >' (\sigma(i+1), \sigma^c(i+1))$
(for 
$i=1, 2, \cdots, n-1$) 
and 
$\sigma^c(n)=p-1$
(for 
$i=n$). 
We denote by 
$d'(\sigma)$ 
the number of dash-descents of 
$\sigma$. 
For
$p=1$, 
$d(\sigma) = d'(\sigma)$. 
\begin{remark}
Let us denote the corresponding descent statistics by 
\beq
E_p(n, k) := 
\sharp \{ \sigma \in G_{p, n} \, | \, 
d(\sigma) = k \},
\quad
F_p(n, k) := 
\sharp \{ \sigma \in G_{p, n} \, | \, 
d'(\sigma) = k \}. 
\eeq
Then 
$F_p$ 
satisfies the recursion relation 
\beq
F_p (n,k) = 
(pk+p-1) F_p(n-1, k)
+
\{ p(n-k)+1 \} 
F_p (n-1, k-1)
\eeq
from which we have
\beq
F_p(n,k) = E_p(n, n-k),
\quad
p \ne 1.
\eeq
In this sense 
$d'(\sigma)$
is a ``reverse" of 
$d(\sigma)$. 
\end{remark}
Using the notion of 
dash-descent, 
we can state our theorem on the relation between 
$(-b, n, p)$-carries proccess 
$\{ \kappa_r^- \}$ 
with 
$\kappa^-_0=0$ 
and the sequence 
$\{ \sigma_r \}$
induced by 
$(b, n, p)$-shuffle.
%
\begin{theorem}
\label{generalminusshuffle}
Let 
$p \in {\bf N}$, 
$p \ne 1$
and 
$b \equiv -1 \pmod p$. 
Then 
$\{ \kappa_r^- \} \stackrel{d}{=}\{ d_r \}$, 
where 
\beq
d_r
\stackrel{}{:=}
\cases{
n - d'(\sigma_r) 
& $(r : odd)$ \cr
d(\sigma_r) 
& $(r : even)$ \cr
}
\eeq
\end{theorem}
For 
$p=1$, 
$n - d(\sigma_r)$
in the definition of 
$d_r$ 
should be replaced by 
$n-1-d(\sigma_r)$. 
Theorem \ref{generalminusshuffle}
implies that 
$d_r$
is a Markov chain on 
${\cal C}_p (n)$. 

So far 
we studied the 
$(\pm b, n, p)$-process 
and showed that, if 
$p \in {\bf N}$, 
they are related to 
the descent statistics(Remark \ref{descent}), 
the Foulkes character table(Remark \ref{Foulkes}), 
and the sequence induced by  
$(+b, n, p)$-shuffle on   
$G_{p, n}$
(Theorems \ref{generalshuffle}, \ref{p1}, \ref{generalminusshuffle}).
It is desirable to consider the counterparts (if any) for 
$p \notin {\bf N}$. 

In the following sections, 
we prove theorems stated above. 
%

%
\section{Right Eigenvectors and Their Duality}
\subsection{Right eigenvectors}
{\it Proof of Theorem \ref{right}}\\
It suffices to show 
$R_p = L_p^{-1}$. 
In what follows 
we omit $n$-dependence and compute
\beq
\sum_{m=0}^n
u_{im}^{(p)}
v_{m, j}^{(p)} 
&=&
\sum_{k=i}^n
\frac {1}{k!}
\left( \begin{array}{c}
n-i \\ n-k
\end{array} \right)
\sum_{r=0}^j
(-1)^r
\left( \begin{array}{c}
n+1 \\ r
\end{array} \right)
\sum_{l=0}^k
\frac {s(k, l)}{p^l}
\\
&& 
\times
\sum_{m=n-l}^n
(-1)^{n-m-l}
\left( \begin{array}{c}
l \\ n-m
\end{array} \right)
\{ p(j-r) + 1 \}^{n-m}.
\eeq
Here 
we changed the order of summation noting that 
$n-m \le l \le k$
implies 
$n-l \le m \le n$.
By the binomial theorem, 
we have
\[
\sum_{m=n-l}^n
(-1)^{n-m-l}
\left( \begin{array}{c}
l \\ n-m
\end{array} \right)
\{ p(j-r) + 1 \}^{n-m}
=
\{ p (j-r) \}^l.
\]
It follows that 
\beq
\sum_{m=0}^n
u_{im}^{(p)}
v_{m, j}^{(p)}
&=&
\sum_{k=i}^n
\frac {1}{k!}
\left( \begin{array}{c}
n-i \\ n-k
\end{array} \right)
\sum_{r=0}^j
(-1)^r 
\left( \begin{array}{c}
n+1 \\ r
\end{array} \right)
\sum_{l=0}^k
s(k,l)(j-r)^l.
\eeq
It is well known that 
$s(n, k)$
satisfies the following equation
\[
x(x-1)(x-2) \cdots (x - n+1)
=
n!
\left( \begin{array}{c}
x \\ n
\end{array} \right)
=
\sum_{k \ge 0} s(n, k) x^k
\]
which implies 
\[
\sum_{l=0}^k
s(k,l)(j-r)^l
=
k!
\left( \begin{array}{c}
j-r \\ k
\end{array} \right).
\]
Therefore
\beq
\sum_{m=0}^n
u_{im}^{(p)}
v_{m, j}^{(p)}
&=&
\sum_{r=0}^j
(-1)^r
\left( \begin{array}{c}
n+1 \\ r
\end{array} \right)
\sum_{k=i}^n
\left( \begin{array}{c}
n-i \\ n-k
\end{array} \right)
\left( \begin{array}{c}
j- r \\ k
\end{array} \right)
\\
&=&
\sum_{r=0}^j
(-1)^r
\left( \begin{array}{c}
n+1 \\ r
\end{array} \right)
\left( \begin{array}{c}
n-i + j -r \\ n
\end{array} \right).
\eeq
Here we used the formula 
$
\left( \begin{array}{c}
a+b \\ n
\end{array} \right)
=
\sum_i
\left( \begin{array}{c}
a \\ i
\end{array} \right)
\left( \begin{array}{c}
b \\ n-i
\end{array} \right).
$
Since RHS 
is equal to the coefficient of 
$x^j$ 
in
$(1-x)^{n+1}
\times
x^i (1-x)^{-(n+1)}$, 
we have 
\[
\sum_{m=0}^n
u_{im}^{(p)}
v_{m, j}^{(p)}
 = \delta_{ij}.
\]
\QED\\

\noindent
{\it Proof of Theorem \ref{StirlingFrobenius}}\\
We will make use of following formulas : 
\begin{eqnarray}
&&
\left(
\begin{array}{c}
n \\ k
\end{array}
\right)
=
\left(
\begin{array}{c}
n-1 \\ k
\end{array}
\right)
+
\left(
\begin{array}{c}
n-1 \\ k-1
\end{array}
\right)
\label{binomialone}
\\
&&
\frac nk
\left(
\begin{array}{c}
n-1 \\ k-1
\end{array}
\right)
=
\left(
\begin{array}{c}
n \\ k
\end{array}
\right)
\label{binomialtwo}
\\
&&
s(k,l)
=
s(k-1, l-1) - (k-1)s(k-1, l)
\label{Stirlingrecursion}
\end{eqnarray}
$w_0(0) = 1$
is clear.
Our strategy is 
to decompose 
$w_j(n)$ 
such as : 
\beq
w_j(n) = A + B, \;
B = C+D, \;
D = E-F, \;
E = G+H
\eeq
and to show 
\beq
A = pn w_j(n-1), \;
G = - w_j(n-1), \;
H = w_{j-1}(n-1), \;
F = C.
\eeq
We first apply (\ref{binomialone}).
\beq
w_j(n)
&=&
n! p^n
\sum_{k=0}^n
\sum_{l=j}^k
\frac {s(k,l)(-1)^{j-l}}{k!p^l}
\left(
\begin{array}{c}
l \\ j
\end{array}
\right)
\left\{
\left(
\begin{array}{c}
n-1 \\ k
\end{array}
\right)
+
\left(
\begin{array}{c}
n-1 \\ k-1
\end{array}
\right)
\right\}
\\
&=:& A + B
\eeq
Then we have 
$A = pn w_j(n-1)$. 
Next we apply 
(\ref{binomialtwo}), (\ref{binomialone}) 
to $B$. 
\beq
B
&=&
(n-1)! p^n
\sum_{k=0}^n
\sum_{l=j}^k
\frac {s(k,l)(-1)^{j-l}}{(k-1)!p^l}
\left(
\begin{array}{c}
l \\ j
\end{array}
\right)
\left(
\begin{array}{c}
n \\ k
\end{array}
\right)
\\
&=&
(n-1)! p^n
\sum_{k=0}^n
\sum_{l=j}^k
\frac {s(k,l)(-1)^{j-l}}{(k-1)!p^l}
\left(
\begin{array}{c}
l \\ j
\end{array}
\right)
\left\{
\left(
\begin{array}{c}
n-1 \\ k
\end{array}
\right)
+
\left(
\begin{array}{c}
n-1 \\ k-1
\end{array}
\right)
\right\}
\\
&=:& C + D
\eeq
Letting 
$k'=k-1$, $l'=l-1$
in 
$D$, 
and then apply 
(\ref{Stirlingrecursion}). 
\beq
D
&=&
(n-1)! p^{n-1}
\sum_{k'=0}^{n-1}
\sum_{l'=j-1}^{k'}
\frac {s(k'+1,l'+1)(-1)^{j-l'-1}}{k'!p^{l'}}
\left(
\begin{array}{c}
l'+1 \\ j
\end{array}
\right)
\left(
\begin{array}{c}
n-1 \\ k'
\end{array}
\right)
\\
&=&
(n-1)! p^{n-1}
\sum_{k'=0}^{n-1}
\sum_{l'=j-1}^{k'}
\frac {
\left(
s(k',l')-k's(k', l'+1)
\right)
(-1)^{j-l'-1}
}
{k'!p^{l'}}
\times
\\
&& \qquad\times
\left(
\begin{array}{c}
l'+1 \\ j
\end{array}
\right)
\left(
\begin{array}{c}
n-1 \\ k'
\end{array}
\right)
\\
&=:& E - F
\eeq
We apply 
(\ref{binomialone}) to $E$. 
\beq
E
&=&
(n-1)! p^{n-1}
\sum_{k=0}^{n-1}
\sum_{l=j-1}^{k}
\frac {s(k,l)(-1)^{j-l-1}}{k!p^{l}}
\left(
\begin{array}{c}
l+1 \\ j
\end{array}
\right)
\left(
\begin{array}{c}
n-1 \\ k
\end{array}
\right)
\\
&=&
(n-1)! p^{n-1}
\sum_{k=0}^{n-1}
\sum_{l=j-1}^{k}
\frac {s(k,l)(-1)^{j-l-1}}{k!p^{l}}
\left\{
\left(
\begin{array}{c}
l \\ j
\end{array}
\right)
+
\left(
\begin{array}{c}
l \\ j-1
\end{array}
\right)
\right\}
\left(
\begin{array}{c}
n-1 \\ k
\end{array}
\right)
\\
&=:& G + H
\eeq
Then we have 
$G = - w_j (n-1)$, 
$H = w_{j-1}(n-1)$. 
Letting 
$l'' = l'+1$ 
in 
$F$, we have 
$F = C$. 
\QED
%

%
\subsection{Some examples}
In this subsection
we give some examples of 
$R_p$. 
To simplify the notation, let 
$\tilde{R}_p := n! p^n R_p$.
In the examples below, 
we take 
$n=3$. 
\beq
\tilde{R}_1 &=& \left(
\begin{array}{ccc}
1 & 3 & 2 \\
1 & 0 & -1 \\
1 & -3 & 2 
\end{array}
\right), 
\quad
\tilde{R}_2 = \left(
\begin{array}{cccc}
1 & 9 & 23 & 15 \\
1 & 3 & -1 & -3 \\
1 & -3 & -1 & 3 \\
1 & -9 & 23 & -15
\end{array}
\right), 
\\
\tilde{R}_3
&=&
\left(
\begin{array}{cccc}
1 & 15 & 66 & 80 \\
1 & 6 & 3 & -10 \\
1 & -3 & -6 & 8 \\
1 & -12 & 39 & -28 
\end{array}
\right), 
\quad
\tilde{R}_{3/2}
=
\left(
\begin{array}{cccc}
1 & 6 & \frac {39}{4} & \frac 72 \\
1 & \frac 32 & - \frac 32 & -1 \\
1 & -3 & \frac 34 & \frac 54 \\
1 & - \frac {15}{2} & \frac {33}{2} & - 10
\end{array}
\right).
\eeq
The 1st row  vectors (in reversed order)
$(2, 3, 1)$, 
$(15, 23, 9, 1)$, $(80, 66, 15, 1)$
of 
$\tilde{R}_1, \tilde{R}_2$, $\tilde{R}_3$ 
are the Stirling Frobenius numbers of parameter 
$p$ 
for 
$p=1, 2, 3$
respectively. 

Each column vectors of 
$\tilde{R}_2$
is symmetric up to sign, and 
if we multiply
each column vectors of 
$\tilde{R}_{3/2}$
by 
$1, (-2), (-2)^2, (-2)^3$, 
and then turn them over, 
we obtain 
$\tilde{R}_3$. 
They are the duality relation mentioned in Remark \ref{dualityR}.
%

\section{Relation to Riffle Shuffles : $(+b)$-case}
%
\subsection{Preliminaries}
First of all, 
we shall recall the definition of 
$*$-map 
introduced in 
\cite{DF1} 
and study its relation to 
$(b, n, p)$-shuffle.
Let 
$
{\cal D}(b)
:=
\{0, 1, \cdots, b-1\}
$
and for 
$b_1, b_2 \in {\bf N}$, 
let 
$
{\bf A}_1 
:= ^t
(a_1^{(1)}, \cdots, a_n^{(1)}) \in {\cal D}(b_1)^n
$, 
$
{\bf A}_2 
:= ^t
(a_1^{(2)}, \cdots, a_n^{(2)}) \in {\cal D}(b_2)^n.
$
The star map 
$({\bf A}_2, {\bf A}_1) \mapsto ({\bf A}_2{\bf A}_1)^*=
\{ a_{i, *}^{(j)} \}_{i=1, \cdots, n}^{j=1,2}
\in 
{\cal D}(b_2) \times {\cal D}(b_1)$
is defined as follows : 
$a_{i, *}^{(1)} := a_{i}^{(1)}$, 
$i=1, \cdots, n$
and 
$a_{1,*}^{(2)}, \cdots, a_{n,*}^{(2)}$
is defined by sorting 
$a_1^{(2)}, \cdots, a_n^{(2)}$
along the order of 
$a_{1, *}^{(1)}, \cdots, a_{n, *}^{(1)}$.
In the example below, we take 
$b_1 = 4$,  
$b_2 = 7$, 
and 
$n=6$. 
\beq
\begin{array}{ccccc}
A_2 & A_1 && (A_2 & A_1)^* \\ \hline
5 & 1 && 0 & 1  \\
0 & 3 && 3 & 3 \\
3 & 2 && 4 & 2 \\
4 & 0 && 5 & 0 \\
6 & 1 && 3 & 1  \\
3 & 2 && 6 & 2  \\   
\end{array} 
\eeq
Furthermore, let 
${\bf A}_j \in {\cal D}(b_j)^n$, 
$j=1, \cdots, N$, 
and suppose that we have defined 
$({\bf A}_k \cdots {\bf A}_1)^*
=
\{ a_{i, *}^{(j)} \}_{i=1, \cdots, n}^{j=1, \cdots, k}$.
Let 
\beq
\tilde{\bf A}_{i, *}^{(k)}
:=
(a_{i, *}^{(k)}, \cdots, a_{i, *}^{(1)})
\in 
{\cal D}(b_k) \times \cdots \times {\cal D}(b_1), 
\quad
i=1,\cdots, n.
\eeq
Then
$\{ a_{i, *}^{(k+1)} \}_{i=1}^n$
is defined by sorting  
$a_{1}^{(k+1)}, \cdots, a_{n}^{(k+1)}$
along the order of 
$\tilde{\bf A}^{(k)}_{1, *}, \cdots, \tilde{\bf A}^{(k)}_{n, *}$
in the lexicographic order in 
${\cal D}(b_k) \times \cdots \times {\cal D}(b_1)$. 
Repeating 
this processes for 
$k=1, \cdots, N$, 
we define the star map 
$({\bf A}_N \cdots {\bf A}_1)^*
=
\{ a_{i, *}^{(j)} \}_{i=1, \cdots, n}^{j=1, \cdots, N}$
for given 
$( {\bf A}_N, \cdots, {\bf A}_1)
=
\{ a_i^{(j)} \}_{i=1, \cdots, n}^{j=1, \cdots, N}
$, 
which is a bijection on 
${\cal D}(b_N)^n \times \cdots \times {\cal D}(b_1)^n$. 
$*$-map 
gives the GSR representation of the composition of 
$(b, n, p)$-shuffles. 
\begin{lemma}\mbox{}\\
\label{star}
Let 
$b_1 \equiv 1 \pmod p$, 
and let 
${\bf A}_j :=\{ a_i^{(j)} \}_{i=1, \cdots, n}^{j=1, 2}$.
Then 
\beq
{\bf (A_2 A_1)}^{\sharp} &:=& ^t(a_1, \cdots, a_n)
\\
\mbox{where }\quad
a_i &:=& a_{i, *}^{(2)} \cdot b_1 
+
a_{i, *}^{(1)}
\in
{\cal D}(b_1 b_2), 
\quad
i = 1, 2, \cdots, n
\eeq
satisfies 
\beq
\pi_{b_1 \cdot b_2}[{\bf (A_2 A_1)}^{\sharp}]
=
\pi_{b_2}[{\bf A}_2] \circ \pi_{b_1}[{\bf A}_1].
\eeq
\end{lemma}
\begin{proof}\mbox{}\\
Let 
$\sigma_j := \pi_{b_j}[{\bf A}_j]$  
($j=1, 2$), 
$\sigma := \sigma_2 \circ \sigma_1$, 
and 
$\mu := 
\pi_{b_2 \cdot b_1}
[ ({\bf A}_2 {\bf A}_1)^{\sharp} ]
\in G_{p, n}$. 
We note that 
the order of elements in 
$({\bf A}_2 {\bf A}_1)^{\sharp} \in {\cal D}(b_2 b_1)^n$
is equivalent to that of 
$\tilde{\bf A}^{(2)}_{1, *}, \cdots, \tilde{\bf A}^{(2)}_{n, *}
\in {\cal D}(b_2) \times {\cal D}(b_1)$
in the lexicographic order.
Then by the definition of  
$*$-map, 
we have 
\beq
\mu(i) &=& (\sigma_2 \circ \sigma_1)(i) = \sigma(i).
\eeq
On the other hand, since 
\beq
a_i &=& a_{i, *}^{(2)} \cdot b_1 
+
a_{i, *}^{(1)}
\equiv 
a_{i,*}^{(2)} + a_{i,*}^{(1)} \pmod p, 
\eeq
and since, by the definition of 
$*$-map, 
\[
\sigma^c(i) \equiv a_{i,*}^{(2)} + a_{i,*}^{(1)} \pmod p, 
\]
we have 
$\mu^c(i) \equiv a_i \equiv \sigma^c(i) \pmod p$. 
\QED
\end{proof}
\begin{lemma}\mbox{}\\
\label{composition}
(1)
The composition of one step of 
$(b_1, n, p)$-carries process 
and one step of 
$(b_2, n, p)$-carries process
have the same distribution as that of 
$(b_2 \cdot b_1, n, p)$-carries process. 
\\
(2)
Let 
$b_1 \equiv 1 \pmod p$.
Then the composition of one step of 
$(b_1, n, p)$-shuffle 
and one step of 
$(b_2, n, p)$-shuffle
have the same distribution as that of 
$(b_2 \cdot b_1, n, p)$-shuffle. 
\end{lemma}
\begin{proof}\mbox{}\\
(1)
This follows either from Theorem \ref{left}, or from the definition of 
$(b, n, p)$-process
and the fact that 
$A_+(b_2) b_1 + A_+(b_1) = A_+(b_2 b_1)$. \\
(2)
Let 
${\bf A}_j \in {\cal D}(b_j)^n$, $j=1,2$. 
It suffices to construct 
${\bf A} \in {\cal D}(b_1 \cdot b_2)^n$
from 
$({\bf A}_2, {\bf A}_1) 
\in {\cal D}(b_2)^n \times {\cal D}(b_1)^n$ 
bijectively such that 
$\pi_{b_1 \cdot b_2} [ {\bf A} ] 
=
\pi_{b_2} [ {\bf A}_2 ] \circ \pi_{b_1} [ {\bf A}_1 ]$. 
By Lemma \ref{star}, 
we have only to take 
${\bf A} := ({\bf A}_2 {\bf A}_1)^{\sharp}$. 
\QED
\end{proof}
Next, 
we need to introduce notions of orders and descents on 
${\cal D}(b)^n$.
Let 
${\bf x} := (x_1, \cdots, x_n) \in {\cal D}(b)^n$. \\
\noindent
(1)
We say that 
${\bf x}$
has a {\bf bar-descent} at 
$i$
if and only if
\\
(i)
$x_i > x_{i+1}$ ($i=1, 2, \cdots, n-1$), 
and 
(ii)
$x_n > A_+(b)'=\frac {b-1}{p}$ ($i=n$). 
\\
We denote by 
$\bar{d}_b({\bf x})$
the number of bar-descents of 
${\bf x}$. 
It 
will be important to make 
$b$-dependence explicit. 
\\
\noindent
(2)
We define an order
$\prec$
on 
${\cal D}(b)$ 
as follows : 
writing 
$x = j_x p + r_x \in {\cal D}(b)$
with 
$r_x =0, \cdots, p-1$, 
we say 
$x \prec y$
if and only if 
$(j_x, r_x) < (j_y, r_y)$. 
where an order on 
$(j, r)'s$
are given 
by the following inequalities
($c=A_+(b)'$). 
\beq
&&
(0,0) < (1,0) < \cdots < (c-1,0)<(c,0) \\
<&&
(0, p-1) < (1, p-1) < \cdots < (c-1, p-1)
\\
<&&
(0, p-2) < (1, p-2) < \cdots < (c-1, p-2)
\\
<&&
 \cdots < 
\\
<&&
(0, 1) < (1, 1) < \cdots < (c-1, 1).
\eeq
We say that 
${\bf x} := (x_1, \cdots, x_n) \in {\cal D}(b)^n$
has a {\bf tilde-descent} at 
$i$ 
if and only if 
(i)
$x_i \succ x_{i+1}$ 
($i=1, 2, \cdots, n-1$), 
and 
(ii)
$x_n \not\equiv 0 \pmod p$. \\
We denote by 
$\tilde{d}_b({\bf x})$
the number of tilde-descents of 
${\bf x}$. \\

Let 
$f_b : {\cal D}(b) \to {\cal D}(b)$
be a bijection given by 
\[
f_b (x) := px 
\quad
\pmod b.
\]
Since 
$b \equiv 1 \pmod p$, 
$x < y$
if and only if 
$f_b(x) \prec f_b(y)$, 
and 
$x > A_+(b)'$
if and only if 
$f_b (x) \not\equiv 0 \pmod p$.
So letting 
\[
f_b({\bf x})
:= (f_b(x_1), \cdots, f_b(x_n))
\in {\cal D}(b)^n
\]
we have
\begin{equation}
\overline{d}_b({\bf x})
=
\tilde{d}_b(f_b({\bf x})).
\label{bartilde}
\end{equation}
By definition of 
$(b, n, p)$-shuffle, 
the number of descents of 
$\pi_b [ f_b({\bf x}) ]$
is equal to tilde-descents of 
$f_b({\bf x})$ : 
\begin{equation} 
\tilde{d}_b(f_b({\bf x}))
=
d(\pi_b[f_b({\bf x})]).
\label{tildeshuffle}
\end{equation}
%

%
\subsection{Coincidence of one step probabilities}
As a first step 
for the proof of Theorem \ref{generalshuffle}, 
we prove the one step probabilities for 
$\kappa_0^+=0$, 
$\sigma_0 = id$
are equal. 
\begin{proposition}
\label{onestep}
For any 
$r \ge 1$, 
and any 
$j \in {\cal C}_p(n):=\{0, \cdots, n\}$, 
we have 
\beq
{\bf P}\left(
\kappa_r = j 
\, | \, 
\kappa_0 = 0
\right)
=
{\bf P}\left(
d(\sigma_r) = j 
\, | \, 
\sigma_0 = id
\right).
\eeq
\end{proposition}
\begin{proof}\mbox{}\\
By
Lemma \ref{composition} 
we may take 
$r=1$ :  
our aim is to show 
$
{\bf P}\left(
\kappa_1 = j 
\, | \, 
\kappa_0 = 0
\right)
=
{\bf P}\left(
d(\sigma_1) = j 
\, | \, 
\sigma_0 = id
\right).
$
Since 
$\kappa_0=0$, 
(\ref{bnp}) 
implies that 
$\kappa_1$
is determined by 
${\bf X}:=(X_1, \cdots, X_n) \in {\cal D}(b)^n$
so that we set  
$\kappa_1 = \kappa({\bf X})$. 
Similarly,
$\sigma_1$
is determined by its GSR representation 
${\bf A}:=(a_1, \cdots, a_n) \in {\cal D}(b)^n$. 
We 
shall construct a bijection  
${\bf X} \mapsto {\bf A}$ 
satisfying 
$\kappa({\bf X}) = d (\pi_b[{\bf A}])$.
The bar map sending 
${\bf X}$
to 
$\overline{\bf X}:=
(\overline{X}_1, \cdots, \overline{X}_n)$
is a bijection on 
${\cal D}(b)^n$
defined by 
\beq
\overline{\bf X}
&:=&
(\overline{X}_1, \cdots, \overline{X}_n)
\in {\cal D}(b)^n
\\
\mbox{where }\;
\overline{X}_i
&:=&
X_1 + \cdots + X_i \in {\cal D}(b)
\pmod b, 
\quad
i=1, 2, \cdots, n.
\eeq
In the definition of 
$(b, n, p)$-carries process, 
in adding 
$(n+1)$-numbers
$X_1, \cdots, X_n, X_{n+1}(=A_+(b))$, 
we have a carry in adding 
$X_{i+1}$ 
to 
$\overline{X}_i$ 
if and only if 
(i)
$\overline{X}_i > \overline{X}_{i+1}
(i=1, \cdots, n-1)$, 
and 
(ii)
$ \overline{X}_n > A_+(b)'
(i=n)$.
Hence
\begin{equation}
d(\kappa({\bf X}))  
=
\overline{d}_b(\overline{\bf X}).
\label{bar}
\end{equation}
Then by 
(\ref{bartilde})-(\ref{bar})
we have 
\begin{equation}
\kappa({\bf X})
=
\overline{d}_b(\overline{\bf X})
=
\tilde{d}_b(f_b(\overline{\bf X}))
=
d(\pi_b[f_b(\overline{\bf X})
])
\label{alltogether}
\end{equation}
and the map  
${\bf X} \mapsto f_b(\overline{\bf X})$
is a bijection on 
${\cal D}(b)^n$.
Therefore 
we complete the proof of Proposition \ref{onestep}. 
\QED
\end{proof}
In Appendix, 
we give 
a different proof of 
Proposition \ref{onestep}
by the generating function method used in \cite{DF2}. 
%
\subsection{Proof of Theorem \ref{generalshuffle}}
We aim to show the following statement :  for any 
$N \in {\bf N}$ 
and for any 
$j_1, \cdots, j_N \in {\cal C}_p(n)$
\begin{eqnarray}
&&
{\bf P}\left(
\kappa_1 = j_1, \kappa_2 = j_2, \cdots, \kappa_N = j_N
\, | \, \kappa_0 = 0
\right)
\nonumber
\\
&&\qquad
=
{\bf P}\left(
d(\sigma_1) = j_1, d(\sigma_2) = j_2, \cdots, d(\sigma_N) = j_N
\, | \, \sigma_0 = id
\right).
\label{aim}
\end{eqnarray}
from which Theorem \ref{generalshuffle} follows.
We prove 
this equality by constructing a bijection as is done in the previous subsection, but the equation
(\ref{alltogether}) 
is realized simultaneously in all  steps. 
In adding 
$n$
numbers of 
$N$ digits with 
$\kappa_0 = 0$, 
\beq
\begin{array}{ccccc}
\kappa_N & \kappa_{N-1} & \cdots & 
\kappa_1 & 0 
\\\hline
& X_1^{(N)} & \cdots & X_1^{(2)} & X_1^{(1)}
\\
& \vdots & & \vdots & \vdots 
\\
& X_n^{(N)} & \cdots & X_n^{(2)} & X_n^{(1)}
\\
& A_+(b) & \cdots & A_+(b) & A_+(b)
\\ \hline
& s^{(N)} & \cdots & s^{(2)} & s^{(1)}
\end{array}
\eeq
$\kappa_1, \cdots, \kappa_N$
are determined by 
$\{ X_i^{(j)} \}_{i=1, \cdots, n}^{j=1, \cdots, N}
\in 
{\cal D}(b)^{Nn}$.
Likewise, since 
$\sigma_0 = id$, 
$\sigma_1, \cdots, \sigma_N$
are determined by their GSR representations  
${\bf A}_1, \cdots, {\bf A}_N$.
We shall construct a bijection 
$\{ X_i^{(j)} \}_{i=1, \cdots, n}^{j=1, \cdots, N}
\mapsto
({\bf A}_N, \cdots, {\bf A}_1)$
on 
${\cal D}(b)^{Nn}$ 
such that 
$\sigma_j = \pi_b [ {\bf A}_j ]
\circ
\cdots
\circ
\pi_b [ {\bf A}_1 ]$
satisfies 
$\kappa_j = d(\sigma_j)$ 
for each 
$j=1, 2, \cdots, N$.

For given 
$a^{(1)}, \cdots, a^{(N)} \in {\cal D}(b)$, 
let 
\begin{equation}
(a^{(N)}, \cdots, a^{(1)})_b
:=
\sum_{j=1}^N
b^{j-1} a^{(j)}
\label{basebexpansion}
\end{equation}
be the number whose base 
$b$
expansion is equal to 
$(a^{(N)}, \cdots, a^{(1)})$.
Given 
$\{ X_i^{(j)} \}_{i=1, \cdots, n}^{j=1, \cdots, N}
\in {\cal D}(b)^{Nn}$, 
we define 
$\{ \overline{X}_i^{(j)} \}_{i=1, 2, \cdots, n}^{j=1, 2, \cdots, N}$
by 
\begin{equation}
(
\overline{X}^{(N)}_i, \cdots, 
\overline{X}^{(1)}_i 
)_b
:=
\sum_{k=1}^i
( X^{(N)}_k, \cdots, X^{(1)}_k )_b, 
\quad
\mbox{mod} \; b^N.
\label{barmap}
\end{equation}
For each 
$j=1, 2, \cdots, N$, 
let 
\begin{equation}
\overline{\bf X}^{(j)}
:=
\left(
\begin{array}{c}
(
\overline{X}^{(j)}_1, 
\cdots, 
\overline{X}^{(1)}_1
)_b
\\
\vdots \\
(
\overline{X}^{(j)}_n, 
\cdots, 
\overline{X}^{(1)}_n
)_b
\\
\end{array}
\right)
\in 
{\cal D}(b^j)^n
\label{barvector}
\end{equation}
be the truncation of 
$\{ 
(\overline{X}_i^{(N)}, \cdots, \overline{X}_i^{(1)})_b
\}_{i=1}^n$
up to 
$j$-th digits. 
Then by Lemma \ref{composition}(1), 
we have 
\begin{equation}
\kappa_j = \overline{d}_{b^j}(\overline{\bf X}^{(j)}), 
\quad
j=1, \cdots, N.
\label{oneone}
\end{equation}
Next we define 
$\{ Y_i^{(j)} \}_{i=1, 2, \cdots, n}^{j=1, 2, \cdots, N}$
by 
\begin{equation}
(Y^{(N)}_i, \cdots, Y^{(1)}_i)_b
:=
f_{b^N}
\left(
(\overline{X}^{(N)}_i, \cdots, \overline{X}^{(1)}_i)_b
\right), 
\quad
i=1, \cdots, n.
\label{fb}
\end{equation}
And let
\beq
{\bf Y}^{(j)}
&:=&
\left(
\begin{array}{c}
(
Y^{(j)}_1, 
\cdots, 
Y^{(1)}_1
)_b
\\
\vdots \\
(
Y^{(j)}_n, 
\cdots, 
Y^{(1)}_n
)_b
\\
\end{array}
\right)
\in 
{\cal D}(b^j)^n
\eeq
be the truncation of 
$\{ (Y^{(N)}_i, \cdots, Y^{(1)}_i)_b
 \}_{i=1}^n$
up to 
$j$-th digits.
By definition of 
$f_{b^N}$ 
we have
\beq
f_{b^j}(\overline{\bf X}^{(j)})
:=
\left(
\begin{array}{c}
f_{b^j}
\left(
(\overline{X}^{(j)}_1, \cdots, \overline{X}^{(1)}_1)_b
\right)
\\
\vdots
\\
f_{b^j}
\left(
(\overline{X}^{(j)}_n, \cdots, \overline{X}^{(1)}_n)_b
\right)
\end{array}
\right)
=
{\bf Y}^{(j)}, 
\quad
j=1, \cdots, N.
\eeq
Therefore by 
(\ref{oneone}), 
(\ref{bartilde}), 
(\ref{tildeshuffle}), 
we have 
\beq
\kappa_j
=
\overline{d}_{b^j}(\overline{\bf X}^{(j)})
=
\tilde{d}_{b^j}({\bf Y}^{(j)})
=
d(\pi_{b^j}[{\bf Y}^{(j)}]), 
\quad
j=1, 2, \cdots, N.
\eeq
Taking 
${\bf A}_N, \cdots, {\bf A}_1 \in {\cal D}(b)^n$
such that 
\beq
\{ Y_i^{(j)} \}_{i=1, \cdots, n}^{j=1, \cdots, N}
=
({\bf A}_N \cdots {\bf A}_1)^*, 
\eeq
we have 
\beq
\pi_{b^j}[{\bf Y}^{(j)}]
=
\pi_b [{\bf A}_j]
\circ
\pi_b [{\bf A}_{j-1}]
\circ
\cdots
\circ
\pi_b [{\bf A}_1], 
\quad
j=1, \cdots, N.
\eeq
by Lemma \ref{star}. 
Since the map 
$\{ X_i^{(j)} \}_{i=1, \cdots, n}^{j=1, \cdots, N}
\mapsto 
({\bf A}_N, \cdots, {\bf A}_1)$
is a bijection on 
${\cal D}(b)^{Nn}$, 
we complete the proof of 
Theorem \ref{generalshuffle}. 
\QED\\
\noindent
{\bf Example}\\
Let 
$p=3$,  
$b = 3 \cdot 2 + 1 = 7$, 
$n = 4$, 
and 
$N=3$.
In this case 
$A_+(b) = 4$. 
Suppose that the first three steps of a 
$(7, 4, 3)$-carries process is given by 
%
\beq
\begin{array}{cccc}
2 & 3 & 3 & \\\hline
&3 & 5 & 4 \\
&0 & 2 & 5 \\
&4 & 4 & 6 \\
&0 & 3 & 2 \\
& 4 & 4 & 4 \\\hline
& 0 & 0 & 0
\end{array}
\eeq
Then 
$\kappa_1 = 3$, 
$\kappa_2 = 3$, 
$\kappa_3 = 2$.
Applying 
bar map, 
$f_{b^3}$, 
$\star^{-1}$, 
and 
$\pi_b$
we have 
\beq
&&
\stackrel{bar \; map}{\Longrightarrow}
\quad
\begin{array}{ccc}\hline
3 & \textcolor{red}{5} & \textcolor{red}{4} \\
\textcolor{red}{4} & 1 & \textcolor{red}{2} \\
1 & \textcolor{red}{6} & 1 \\
\textcolor{red}{2} & \textcolor{red}{2} & \textcolor{red}{3} \\\hline
\end{array}
\quad
\stackrel{f_{b^3}}{\Longrightarrow}
\quad
\begin{array}{ccc}\hline
4 & \textcolor{red}{2} & \textcolor{red}{5} \\
\textcolor{red}{5} & 3 & \textcolor{red}{6} \\
5 & \textcolor{red}{4} & 3 \\
\textcolor{red}{0} & \textcolor{red}{0} & \textcolor{red}{2} \\
\hline
\end{array}
\quad
\stackrel{\star^{-1}}{\Longrightarrow}
\quad
\begin{array}{ccc}
{\bf A}_3 & {\bf A}_2 & {\bf A}_1 \\\hline
0 & 0 & 5 \\
4 & 4 & 6 \\
5 & 2 & 3 \\
5 & 3 & 2 \\\hline
\end{array}
\\
&&
\quad
\stackrel{\pi}{\Longrightarrow}
\quad
\begin{array}{ccc}
\pi_7({\bf A}_3) & \pi_7({\bf A}_2) & \pi_7({\bf A}_1) \\\hline
(1,0) & (1,0) & (3,2) \\
(2,1) & (4,1) & (4,0) \\
(3,2) & (2,2) & (2,0) \\
(4,2) & (3,0) & (1,2)\\\hline
\end{array}
\eeq
Then 
$\sigma_1 := \pi_7 [{\bf A}_1]$
$\sigma_2 := \pi_7[{\bf A}_2] \circ \pi_7[{\bf A}_1]$, 
$\sigma_3 := \pi_7[{\bf A}_3] \circ \pi_7[{\bf A}_2] \circ \pi_7[{\bf A}_1]$
satisfy 
$d(\sigma_1) =  3$, 
$d(\sigma_2) = 3$, 
$d(\sigma_3) = 2$. 
\beq
\begin{array}{ccc}
\sigma_3 & \sigma_2 & \sigma_1 \\\hline
(2,2) & \textcolor{red}{(2,1)} & \textcolor{red}{(3,2)} \\
\textcolor{red}{(3,2)} & (3,0) & \textcolor{red}{(4,0)} \\
(4,0) & \textcolor{red}{(4,1)} & (2,0) \\
\textcolor{red}{(1,2)} & \textcolor{red}{(1,2)} & \textcolor{red}{(1,2)} 
\\\hline
\end{array}
\eeq
%

\section{Relation to Riffle Shuffles : $(-b)$-case}
\subsection{Simpler case}
In this subsection 
we prove Theorem \ref{p1}.
We begin by proving 
some symmetry properties of 
$P_p^{\pm}$ : 
\begin{proposition}
\label{plusminus1}
%
%
\beq
&(0)&\;
P_1^+(i,j)
=
P_1^+ (n-1-i, n-1-j), 
\quad
i, j \in {\cal C}_1(n),
\\
&(1)&\;
P_1^-(i,j)
=
P_1^+(i, n-1-j),
\quad
i, j \in  {\cal C}_1(n),
\\
&(2)&\;
P_2^-(i,j)
=
P_2^+(i, n-j), 
\quad
i, j \in {\cal C}_2(n),
\\
&(3)&\;
P^{\pm}_p(i,j)
=
P^{\pm}_{p^*}(n-i, n-j), 
\quad
p>1, \;
i, j \in {\cal C}_p(n).
\eeq
\end{proposition}
\begin{proof}
(0)
follows from (\ref{bnp}). 
By definition, 
$b^n P_p^{\pm}(i,j)$
is equal to the number of 
$(X_1, \cdots, X_n, Y) \in {\cal D}(b)^{n+1}$ 
such that 
$X_1 + \cdots + X_n + Y = B_p^{\pm}(i,j)$. 
Thus 
we can show 
(1) - (3) 
from the following relations : 
$B_1^- (i,j) = B_1^+ (i, n-1-j)$, 
$B_2^- (i,j) = B_2^+ (i, n-j)$, 
$(n+1)(b-1) - B^+_p(i, j)
=
B^+_{p^*}(n-i, n-j)$. 
\QED
\end{proof}
\noindent
{\it Proof of Theorem \ref{p1}}\\
Let 
$r(i)$
be the ``reverse" of 
$i$ : 
$r(i) := n-1-i$, 
$i \in {\cal C}_1(n)$.
By 
Proposition \ref{plusminus1}, 
we have 
$P_1^- (i, j)
=
P_1^+ (i, r(j))
=
P_1^+ (r(i), j)$.
Hence, denoting by  
$\{ \kappa^{\pm}_r \}_{r=0}^{\infty}$
the 
$(\pm b, n, 1)$-carries processes 
respectively, we have
\beq
&&
{\bf P} \left(
\kappa^-_1 = j_1, \kappa^-_2 = j_2, 
\kappa^-_3=j_3, 
\cdots, \kappa^-_N = j_N
\, | \, 
\kappa^-_0 = 0
\right)
\\
&&=
{\bf P} \left(
\kappa^+_1 = r(j_1), \kappa^+_2 = j_2, 
\kappa^+_3=r(j_3), 
\cdots, \kappa^+_N = r^N(j_N)
\, | \, 
\kappa^+_0 = 0
\right).
\eeq
Let 
$\{ \sigma_r \}_{r=0}^{\infty}$
be the sequence induced by 
$(b, n, 1)$-shuffle 
with 
$\sigma_0 = id$. 
Then by Theorem \ref{generalshuffle} with $p=1$, 
\begin{eqnarray}
&&
{\bf P}\left(
\kappa^+_1 = r(j_1), \kappa^+_2 = j_2, 
\cdots, 
\kappa^+_N = r^N(j_N)
\, | \, 
\kappa^+_0 = 0
\right)
\nonumber
\\
&=&
{\bf P}\left(
d(\sigma_1) = r(j_1), d(\sigma_2) = j_2, 
\cdots, 
d(\sigma_N) = r^N(j_N)
\, | \,
\sigma_0 = id
\right).
\label{shuffle1}
\end{eqnarray}
Let 
$\sigma^R \in S_n$
be the ``reverse" of 
$\sigma \in S_n$ 
defined by 
$(\sigma^R)(j) := n+1-j$, 
$j=1, 2, \cdots, n$.
and consider a Markov chain 
$\{ (\sigma_r)^R \}_{r=0}^{\infty}$.
Noting 
$d(\sigma^R) = n-1-d(\sigma) = r(d(\sigma))$, 
the descent of 
$\sigma_r$
and that of 
$ (\sigma_r)^R$
evolve as follows.  
%
%
\beq
\;d(\sigma_0)&=&0 
\quad
\textcolor{red}{\longrightarrow}
\quad
d(\sigma_1)=r(j_1)
\longrightarrow
\quad\;
d(\sigma_2)=j_2
\quad
\textcolor{red}{\longrightarrow}
\quad\;\;
d(\sigma_3)=r(j_3)
\\
&&
\qquad\qquad\qquad
\textcolor{red}{\downarrow}
\qquad\qquad\qquad\quad\quad
\textcolor{red}{\uparrow}
\qquad\qquad\qquad\qquad\quad
\textcolor{red}{\downarrow}
\\
d((\sigma_0)^R)&=&r(0)
\longrightarrow
\;
d((\sigma_1)^R)=j_1
\;\;
\textcolor{red}{\longrightarrow}
\;
d((\sigma_2)^R)=r(j_2)
\longrightarrow
d((\sigma_3)^R)=j_3
\eeq
Since 
$\tilde{\sigma}_r = 
(\sigma_r)^R (r : odd), 
=\sigma_r (r : even)$, 
we have 
\begin{eqnarray}
&&
{\bf P}\left(
d(\sigma_1) = r(j_1), d(\sigma_2) = j_2, 
\cdots, 
d(\sigma_N) =r^N (j_N)
\, | \,
\sigma_0 = id
\right)
\nonumber
\\
&&=
{\bf P}\left(
d(\tilde{\sigma}_1) = j_1, d(\tilde{\sigma}_2) = j_2, 
\cdots, 
d(\tilde{\sigma}_N) = j_N
\, | \,
\sigma_0 =  id
\right).
\label{shuffle2}
\end{eqnarray}
By 
(\ref{shuffle1}), (\ref{shuffle2})
the proof of Theorem \ref{p1} is completed.
\QED
%
\subsection{General case}
The proof of 
Theorem \ref{generalminusshuffle} 
is essentially parallel to that of 
Theorem \ref{generalshuffle}, 
but since 
$b \equiv -1 \pmod p$, 
we need two more steps in the construction of the bijection. 
We begin by preparing some lemmas and more notions of orders and descents. 
\begin{lemma}
\label{reverse}
Let 
$X' := b-1-X$
be the reverse of digit 
$X \in {\cal D}(b)$. 
Then 
\beq
i + X_1 + \cdots + X_n + A_- (b)
=
(n-j) b + s
\eeq
is equivalent to 
\beq
(n-i) + X'_1 + \cdots + X'_n + A'_- (b)
=
j b + s'.
\eeq
\end{lemma}
\begin{lemma}
\label{A}
The composition of 
$j$-steps of  $(-b, n, p)$-carries processes is equivalent to one step of 
$(-b^{2k-1}, n, p)$-carries process
$(j=2k-1)$, 
or one step of 
$(b^{2k}, n, p)$-carries process
$(j=2k)$. 
\end{lemma}
\begin{proof}
It follows either from Theorem \ref{left} or from the following equalities.
\beq
&&
(\overbrace{A_-(b), A_-(b)', A_-(b), \cdots, A_-(b)', A_-(b)}^{2k-1})_b
=
A_- (b^{2k-1})
\\
&&
(\overbrace{A_-(b)', A_-(b), \cdots, A_-(b)', A_-(b)}^{2k})_b
=
A_+ (b^{2k}).
\eeq
LHS of above is defined in (\ref{basebexpansion}).
\QED
\end{proof}
We next define 
the dash-counterparts for 
bar-, and tilde-descents.
Let 
${\bf x} := (x_1, \cdots, x_n) \in {\cal D}(b)^n$. \\
\noindent
(1)
We say that 
${\bf x}$
has a {\bf bar'-descent} at 
$i$
if and only if
\\
(i)
$x_i > x_{i+1}$ ($i=1, 2, \cdots, n-1$), 
and 
(ii)
$x_n > A_-(b)'=\frac {b+1}{p^*}-1$ ($i=n$). 
\\
We denote by 
$\bar{d'}_b({\bf x})$
the number of bar'-descents of 
${\bf x}$. 
\\
\noindent
(2)
We define an order
$\prec'$
on 
${\cal D}(b)$ 
as follows : 
writing 
$x = j_x p + r_x \in {\cal D}(b)$
with 
$r_x =0, \cdots, p-1$, 
we say 
$x \prec' y$
if and only if 
$(j_x, r_x) <' (j_y, r_y)$. 
where an order on 
$(j, r)'s$
are given 
by the following inequalities
($c=A_-(b)$). 
\beq
&&
(0,0) < (1,0) < \cdots < (c-1,0)<(c,0) \\
<&&
(0, 1) < (1, 1) < \cdots < (c, 1)
\\
<&&
(0, 2) < (1, 2) < \cdots < (c, 2)
\\
<&&
 \cdots 
\\
<&&
(0, p-1) < (1, p-1) < \cdots < (c-1, p-1).
\eeq
We say that 
${\bf x} := (x_1, \cdots, x_n) \in {\cal D}(b)^n$
has a {\bf tilde'-descent} at 
$i$ 
if and only if 
(i)
$x_i \succ' x_{i+1}$, 
($i=1, 2, \cdots, n-1$), 
and  
(ii)
$x_n \equiv p-1 \pmod p$. \\
We denote by 
$\tilde{d}'_b({\bf x})$
the number of tilde'-descents of 
${\bf x}$. \\
Since 
$b \equiv -1 \pmod p$, 
$x <' y$
if and only if 
$f_b(x) \prec' f_b(y)$, 
and 
$x > A_-(b)'$
if and only if 
$f_b (x) \equiv p-1 \pmod p$.
Hence we have
\begin{equation}
\overline{d'}_b({\bf x})
=
\tilde{d}'_b(f_b({\bf x})).
\label{bartildedash}
\end{equation}
By definition of 
$(b, n, p)$-shuffle, 
the number of dash-descents of 
$\pi_b [ f_b({\bf x}) ]$
is equal to tilde'-descents of 
$f_b({\bf x})$ 
in 
$\prec'$
order : 
\begin{equation} 
\tilde{d}'_b(f_b({\bf x}))
=
d'(\pi_b[f_b({\bf x})]).
\label{tildeshuffledash}
\end{equation}
Finally 
we state an analogue of 
Lemma \ref{star}. 
For 
$\sigma \in G_{p, n}$, 
let 
\[
\sigma'
:=
\left(
(\sigma(1), -\sigma^c(1)), \cdots, (\sigma(n), -\sigma^c(n))
\right)
\in G_{p, n}.
\]
\begin{lemma}
\label{compositiondash}
Let 
$b_1 \equiv -1 \pmod p$.
Then we have 
\beq
\pi_{b_2b_1}[({\bf A}_2 {\bf A}_1)^{\sharp}]
=
\pi_{b_2} [ {\bf A}_2 ]' \circ \pi_{b_1} [ {\bf A}_1 ].
\eeq
\end{lemma}
{\it Proof of Theorem \ref{generalminusshuffle}}\mbox{}\\
It suffices to show the counterpart of 
(\ref{aim}). 
We suppose 
$N = 2k$
is even ; the proof for odd case is similar. 
Let 
$\kappa^{-'}_i := n - \kappa_i^-$
be the ``reverse" of 
$\kappa_i^-$.
(\ref{bnp}) 
says that, in terms of usual carries process, 
we have to reverse the carry before adding it to the next digit : 
\beq
\begin{array}{cccccc}
\kappa^{-'}_{2k} & \kappa^{-'}_{2k-1} & \cdots  & \kappa^{-'}_2 & \kappa^{-'}_1 & \\
\downarrow &\downarrow & \cdots  & \downarrow & \downarrow & \\
\kappa^{-}_{2k} &\kappa^{-}_{2k-1} & \cdots  & \kappa^{-}_2 & \kappa^{-}_1 & 0 \\\hline
X_{1}^{(2k+1)} & X_{1}^{(2k)} & \cdots  & X_{1}^{(3)} & X_{1}^{(2)} & X_{1}^{(1)} \\
\vdots & \vdots &  &\vdots  & \vdots & \vdots \\
X_{n}^{(2k+1)} & X_{n}^{(2k)} & \cdots  & X_{n}^{(3)} & X_{n}^{(2)} & X_{n}^{(1)} \\
A_-(b) & A_-(b) & \cdots  & A_-(b) & A_-(b) & A_-(b) \\\hline
s_{2k+1} & s_{2k} & \cdots  & s_3 & s_2 & s_1
\end{array}
\eeq
We reverse the numbers in the even digits and use Lemma \ref{reverse}. 
\beq
\stackrel{reverse}{\Longrightarrow}
\quad
\begin{array}{cccccc}
\kappa^{-}_{2k} &\kappa^{-'}_{2k-1} & \cdots & \kappa^{-}_2 & \kappa^{-'}_1 & 0 \\\hline
X_{1}^{(2k+1)} & X_{1}^{(2k)'} & \cdots  & X_{1}^{(3)} & X_{1}^{(2)'} & X_{1}^{(1)} \\
\vdots & \vdots &  &\vdots  & \vdots & \vdots \\
X_{n}^{(2k+1)} & X_{n}^{(2k)'} & \cdots  & X_{n}^{(3)} & X_{n}^{(2)'} & X_{n}^{(1)} \\
A_-(b) &A_-(b)' &\cdots & A_-(b) & A_-(b)' & A_-(b) \\\hline
s_{2k+1} & s'_{2k} & \cdots & s_3 & s'_2 & s_1 
\end{array}
\eeq
Sending it by the bar map(\ref{barmap}), we have 
\beq
\stackrel{bar \; map}{\Longrightarrow}
\quad
\begin{array}{cccccc}
\kappa^{-}_{2k} &\kappa^{-'}_{2k-1} & \cdots & \kappa^{-}_2 & \kappa^{-'}_1 & 0 \\\hline
\overline{X}_{1}^{(2k+1)} & \overline{X}_{1}^{(2k)'} & \cdots  & \overline{X}_{1}^{(3)} & \overline{X}_{1}^{(2)'} & \overline{X}_{1}^{(1)} \\
\vdots & \vdots &  &\vdots  & \vdots & \vdots \\
\overline{X}_{n}^{(2k+1)} & \overline{X}_{n}^{(2k)'} & \cdots  & \overline{X}_{n}^{(3)} & \overline{X}_{n}^{(2)'} & \overline{X}_{n}^{(1)} \\
%
%
s_{2k+1} & s_{2k}' & \cdots & s_3 & s_2' & s_1 
\\\hline
\end{array}
\eeq
By 
Lemmas \ref{composition}, \ref{A}, 
and by the definition of 
bar-, and bar'-descent, we have
\[
\kappa^{-'}_j = \overline{d'}_{b^j} (\overline{\bf X}^{(j)}) 
\quad
j : odd, 
\quad
\kappa^{-}_j = \overline{d}_{b^j} (\overline{\bf X}^{(j)}) 
\quad
j : even.
\]
where 
$\overline{\bf X}^{(j)}$
is defined in 
(\ref{barvector})
but 
$\overline{X}_i^{(j)}$'s 
with even 
$j$'s
are reversed. 
Let 
$\{ Y_i^{(j)} \}$ 
as defined in (\ref{fb}). 
Since 
$b^j \equiv 1 (j : even), -1 (j : odd) 
\pmod p$, 
equations 
(\ref{bartilde}), (\ref{bartildedash})
imply that 
\[
\overline{d'}_{b^j} (\overline{\bf X}^{(j)})
=
\tilde{d'}_{b^j}({\bf Y}^{(j)})
\quad
j : odd, 
\quad
\overline{d}_{b^j} (\overline{\bf X}^{(j)})
=
\tilde{d}_{b^j}({\bf Y}^{(j)})
\quad
j : even.
\]
Take
${\bf A}_{2k}, \cdots, {\bf A}_1 \in {\cal D}(b)^n$ 
such that 
$\{ Y_i^{(j)} \} = 
({\bf A}_{2k} \cdots {\bf A}_1)^*$. 
Then by Lemma \ref{compositiondash} 
\beq
\sigma_r := \cases{
\pi_b [{\bf A}_r] \circ \pi_b [{\bf A}_{r-1}]' \circ \cdots \circ \pi_b [{\bf A}_2]'\circ \pi_b [{\bf A}_1]
&
$(r : odd)$ \cr
\pi_b [{\bf A}_r]' \circ \pi_b [{\bf A}_{r-1}] \circ \cdots \circ \pi_b [{\bf A}_2]'\circ \pi_b [{\bf A}_1]
&
$(r : even)$ \cr
}
\eeq
($r=1, 2, \cdots, 2k$)
have the desired properties. 
\QED\\

\noindent
{\bf Example : }\\
Let 
$p = 3$, $b = 3 \cdot 3 -1 = 8$, 
$N=4$.
Then 
$A_-(b) = 2$, 
$A_-(b)' = 5$. 
In the example below, 
$\kappa^-_1=3$, 
$\kappa^-_2=1$, 
$\kappa^-_3=2$, 
$\kappa^-_4 = 4$. 
\beq
&&
\begin{array}{ccccc}
0 & 2 & 3 & 1 & \\
\downarrow & \downarrow & \downarrow  & \downarrow  & \\
\fbox{4}
&
\fbox{2} 
&
\fbox{1}
&
\fbox{3}
&
\fbox{0}   
\\\hline
& 0 & 4 & 7 & 4 \\
& 1 & 2 & 5 & 3 \\
& 2 & 5 & 4 & 1 \\
& 0 & 3 & 6 & 2 \\ 
& 2 & 2 & 2 & 2 \\\hline
& 7 & 1 & 3 & 4 
\end{array}
\quad
\stackrel{reverse}{\Longrightarrow}
\quad
\begin{array}{cccc}
2 & \fbox{1} & 1 & \fbox{0} \\\hline
7 & 4 & 0 & 4 \\
6 & 2 & 2 & 3 \\
5 & 5 & 3 & 1 \\
7 & 3 & 1 & 2 \\
5 & 2 & 5 & 2 \\\hline
0 & 1 & 4 & 4 
\end{array}
\quad
\stackrel{bar \; map}{\Longrightarrow}
\quad
\begin{array}{cccc}\hline
\textcolor{red}{7} & 4 & 0 & 4 \\
\textcolor{red}{5} & \textcolor{red}{6} & 2 & \textcolor{red}{7} \\
\textcolor{red}{3} & 3 & 6 & 0 \\
\textcolor{red}{2} & \textcolor{red}{6} & \textcolor{red}{7} & 2 \\
0 & 1 & 4 & 4 \\\hline
\end{array}
\\
&&
\quad
\stackrel{f_{b^4}}{\Longrightarrow}
\quad
\begin{array}{cccc}
\hline
6 & 4 & 1 & 4 \\
1 & 3 & 0 & 5 \\
2 & 3 & 2 & 0 \\
0 & 4 & 5 & 6 \\\hline
\end{array}
\quad
\stackrel{\star^{-1}}{\Longrightarrow}
\quad
\begin{array}{cccc} 
{\bf A}_4 & {\bf A}_3 & {\bf A}_2 & {\bf A}_1 \\
\hline
1 & 3 & 2 & 4 \\
2 & 4 & 1 & 5 \\
6 & 3 & 0 & 0 \\
0 & 4 & 5 & 6 \\\hline
\end{array}
\\
&& \\
&&
\quad
\stackrel{\pi}{\Longrightarrow}
\quad
\begin{array}{cccc} 
\pi[{\bf A}_4] & \pi[{\bf A}_3] & \pi[{\bf A}_2] & \pi[{\bf A}_1] \\
\hline
(2,1) & (1,0) & (3,2) & (2,1) \\
(3,2) & (3,1) & (2,1) & (3,2) \\
(4,0) & (2,0) & (1,0) & (1,0) \\
(1,0) & (4,1) & (4,2) & (4,0) \\\hline
\end{array}
\quad
\stackrel{reverse}{\Longrightarrow}
\quad
\begin{array}{cccc} 
\pi[{\bf A}_4]' & \pi[{\bf A}_3] & \pi[{\bf A}_2]' & \pi[{\bf A}_1]
\\\hline
(2,2) & (1,0) & (3,1) & (2,1) \\
(3,1) & (3,1) & (2,2) & (3,2) \\
(4,0) & (2,0) & (1,0) & (1,0) \\
(1,0) & (4,1) & (4,1) & (4,0) \\\hline
\end{array}
\eeq
Thus the corresponding sequence induced by 
$(b, n, p)$-shuffle is : 
\beq
\begin{array}{cccc}
\sigma_4 & \sigma_3 & \sigma_2 & \sigma_1 \\\hline
\textcolor{red}{(4,1)} & (3,1)& (2,0)& (2,1) \\
\textcolor{red}{(2,1)} & \textcolor{blue}{(1,2)} & (1,2) & \textcolor{blue}{(3,2)} \\
\textcolor{red}{(3,2)} & (2,1) & (3,1) & (1,0) \\
\textcolor{red}{(1,2)} & \textcolor{blue}{(4,2)} & \textcolor{red}{(4,1)} & (4,0) \\\hline
\end{array}
\eeq
Indeed, 
$d'(\sigma_1) = 4-3$, 
$d(\sigma_2) = 1$, 
$d'(\sigma_3)=4-2$, 
and 
$d(\sigma_4)=4$. 
%
%
\section{Appendix}
%
\subsection{Expectation and Variance of Carries}
In this subsection 
we prove Theorems \ref{zerostart}, \ref{pistart}.\\

\noindent
{\it Proof of Theorem \ref{zerostart}}\\
We discuss 
$(+b)$-case only ;  
$(-b)$-case can be proved similarly. \\
(1)
(\ref{different})
implies that 
%
\[
\tilde{u}_1 = \{ \tilde{u}_{i, 1} \}_{i=0}^n, 
\quad
\tilde{u}_1(i)
:=
i + \frac 1p - \frac {n+1}{2}
\]
is an right eigenvector of 
$P_p^+$
with eigenvalue 
$b^{-1}$, 
by which we compute
%
%
\beq
{\bf E}[ \kappa_r \, | \, \kappa_0 = i]
&=&
\sum_j
j {\bf P}^{(r)}(i, j) 
\\
&=&
\sum_j
{\bf P}^{(r)} (i, j)
\left(
\tilde{u}_1(j) - \frac 1p + \frac {n+1}{2}
\right)
\\
&=&
\frac {1}{b^r}
\tilde{u}_1(i) 
- \frac 1p + \frac {n+1}{2}
\\
&=&
\frac {1}{b^r}
\left( i + \frac 1p - \frac {n+1}{2} \right)
- \frac 1p + \frac {n+1}{2}.
\eeq
(2)
At first we compute 
${\bf E}[ \kappa_s^2 \, | \, \kappa_0=i ]$.
In order to do that, 
we need a right eigenvector 
$\tilde{u}_2= \{
\tilde{u}_{i, 2} 
\}_{i=0}^n$
with eigenvalue 
$b^{-2}$. 
After some computation, 
we have  
\beq
\tilde{u}_2(i)
&=&
\left(
i + \frac 1p - \frac {n+1}{2}
\right)^2
-
\frac {n+1}{12}.
\eeq
Computing 
similarly as we did in (1), we obtain
\begin{eqnarray}
&&
{\bf E}\left[
\left(
\kappa_r + \frac 1p -
\frac {n+1}{12}
\right)^2
\, \Biggl| \, \kappa_0 = i
\right]
\nonumber
\\
&&\qquad=
\frac {1}{b^{2r}}
\left\{
\left(
i + \frac 1p - \frac {n+1}{2}
\right)^2
-
\frac {n+1}{12}
\right\}
+
\frac {n+1}{12}.
\label{variance1}
\end{eqnarray}
On the other hand 
\begin{eqnarray}
&&
{\bf E}
\left[
\left(
\kappa_r + \frac 1p - \frac {n+1}{12}
\right)^2
\, \Biggl| \, \kappa_0 = i
\right]
=
{\bf E}[ \kappa_r^2 \,|\, 
\kappa_0 = i ]
\nonumber
\\
&& \qquad\quad
+
2 \left(
\frac 1p - \frac {n+1}{12}
\right)
{\bf E}[ \kappa_r \, | \, 
\kappa_0 = i ]
+
\left(
\frac 1p - \frac {n+1}{12}
\right)^2.
\label{variance2}
\end{eqnarray}
By 
(\ref{variance1}), (\ref{variance2})
we have 
\begin{eqnarray}
&&
{\bf E}[ \kappa_r^2 
\, | \, \kappa_0 = i ]
=
\frac {1}{b^{2r}}
\left(
i + \frac 1p - \frac {n+1}{2}
\right)^2
+
\frac {n+1}{12}
\left(
1 - \frac {1}{b^{2r}}
\right)
\nonumber
\\
&& \qquad\quad
-2
\left(
\frac 1p - \frac {n+1}{2}
\right)
{\bf E}[ \kappa_r \,| \, \kappa_0 = i]
-
\left(
\frac 1p - \frac {n+1}{2}
\right)^2.
\label{variance3}
\end{eqnarray}
By squaring 
the equation in Theorem \ref{zerostart}(1), we have 
\begin{eqnarray}
&&
{\bf E}[ \kappa_r \, | \, \kappa_0 = i ]^2
=
\frac {1}{b^{2r}}
\left(
i + \frac 1p - \frac {n+1}{2}
\right)^2
\nonumber
\\
&&  \qquad
+
2 {\bf E}[ \kappa_r \, | \, \kappa_0 = i ]
\left(
- \frac 1p + \frac {n+1}{2}
\right)
-
\left(
- \frac 1p + \frac {n+1}{2}
\right)^2.
\label{variance4}
\end{eqnarray}
Taking difference between
(\ref{variance3})
and
(\ref{variance4}), 
we obtain the conclusion.\\
(3)
We compute
\beq
&&
{\bf E}[ \kappa_s \kappa_{s+r} 
\, | \, \kappa_0 =i ] 
\\
&=&
\sum_j
{\bf E}[ \kappa_{s+r} \, | \, \kappa_s = j]
\cdot j \cdot
{\bf P}(\kappa_s = j \, | \, \kappa_0=i)
\\
&=&
\sum_j
{\bf E}[ \kappa_{r} \, | \, \kappa_0 = j]
\cdot j \cdot
{\bf P}(\kappa_s = j \, | \, \kappa_0=i)
\\
&=&
\sum_j j 
\left\{
\frac {j}{b^r}
+
\left(
\frac {n+1}{2}- \frac 1p
\right)
\left(
1 - \frac {1}{b^r}
\right)
\right\}
{\bf P}(\kappa_s = j \, | \, \kappa_0=i)
\\
&=&
\frac {1}{b^r}
{\bf E}[ \kappa_s^2 \, | \, \kappa_0=i ] 
+
\left(
\frac {n+1}{2}- \frac 1p
\right)
\left(
1 - \frac {1}{b^r}
\right)
{\bf E}[ \kappa_s \, | \, \kappa_0=i ] 
\\
&=&
\frac {1}{b^r}
\frac {n+1}{12}
\left(
1 - \frac {1}{b^{2s}}
\right)
+
\frac {1}{b^r}
{\bf E}[ \kappa_s \, | \, \kappa_0=i ]^2
\\
&& \qquad
+
\left(
\frac {n+1}{2}- \frac 1p
\right)
\left(
1 - \frac {1}{b^r}
\right)
{\bf E}[ \kappa_s \, | \, \kappa_0=i ].
\eeq
In the last equality, 
we used Theorem \ref{zerostart}(2). 
We further compute, 
using Theorem \ref{zerostart}(1). 
\beq
&&
{\bf E}[ \kappa_{s+r} \kappa_s \, | \, \kappa_0=i ]
\\
&=&
\frac {1}{b^r}
\frac {n+1}{12}
\left(
1 - \frac {1}{b^{2s}}
\right)
\\
&& \quad
+ 
\frac {1}{b^r}
\left\{
\frac {1}{b^{s}}
\left(
i + \frac 1p - \frac {n+1}{2}
\right)
+
\left(
\frac {n+1}{2} - \frac 1p 
\right)
\right\}
{\bf E}[ \kappa_s \, | \, \kappa_0=i]
\\
&&\qquad + 
\left(
\frac {n+1}{2} - \frac 1p 
\right)
\left(
1 - \frac {1}{b^r}
\right)
{\bf E}[ \kappa_s \, | \, \kappa_0=i ] 
\\
&=&
\frac {1}{b^r}
\frac {n+1}{12}
\left(
1 - \frac {1}{b^{2s}}
\right)
+
{\bf E}[ \kappa_{r+s} \, | \, \kappa_0=i ] 
\cdot
{\bf E}[ \kappa_{s} \, | \, \kappa_0=i ]
\eeq
which leads to the conclusion. 
\QED\\

\noindent
{\it Proof of Theorem \ref{pistart}}\\
Since 
$(\pm b, n, p)$-process 
is irreducible and aperiodic, 
Markov chain limit theorem yields 
$
{\bf E}[ \kappa_r \, | \, \kappa_0 = i]
\stackrel{r \to \infty}{\to}
{\bf E}_{\pi} [ \kappa_0 ].
$
Hence we can prove 
(1)
(resp. (2))
by taking 
$r \to \infty$
(resp. $s \to \infty$)
in 
Theorem \ref{zerostart}(1)
(resp. in Theorem \ref{zerostart}(3)). 
\QED
%
\subsection{Alternative proof of 
Proposition \ref{onestep}}
Proposition \ref{onestep}
is also proved by the generating function method as is done in 
\cite{DF2}, Theorems 2.3, 4.4. 
We give them for the sake of completeness. 
\begin{lemma}
\label{shufflecounting}
\mbox{}\\
Let 
$b \equiv  1 \pmod p$ 
and let 
$\sigma \in G_{p, n}$. 
The probability 
of obtaining 
$\sigma$
after 
$r$ steps of a 
$(+b, n, p)$-shuffle
is equal to 
\[
{\bf P}\left(
\sigma_r = \sigma \right)
=
b^{-rn}
\left(
\begin{array}{c}
n + \frac {b^r - 1}{p} - d(\sigma^{-1}) \\ n
\end{array}
\right).
\]
\end{lemma}
\begin{proof}\mbox{}\\
By Lemma \ref{composition}, 
we may suppose 
$r=1$. 
Let 
$c = \frac {b-1}{p}$. 
A $(+b, n, p)$-shuffle 
generating 
$\sigma \in G_{p, n}$
is obtained in putting 
$(c - d(\sigma^{-1}))$-
slits in the array 
$\left(
\sigma^{-1}(1, 0), \cdots, 
\sigma^{-1}(n, 0)
\right)$
of 
$n$-elements, 
the numbers of which is equal to 
\[
\left( 
\begin{array}{c}
n + c - d(\sigma^{-1}) \\ n
\end{array} 
\right)
\]
yielding the conclusion for 
$r=1$. 

We give below an example of 
$p=3$,  
$b=3 \cdot 2 + 1 = 7$, 
$n=7$, 
and 
\beq
\sigma = 
\left(
(6,2), (5,1), (2,1), (3,2), (1,0), 
(7,0), (4,0)
\right)
\in G_{3, 7}.
\eeq
${\bf A}=^t
(5,4,1,2,0,6,3)$
is the GSR representation of the 
$(7, 7, 3)$-shuffle 
generating 
$\sigma$. 
$\sigma^{-1}$
has descent at 
$i=3, 6$. 
Since 
$d( \sigma^{-1} ) = 2$, 
there are no other ones generating  
$\sigma$. 
\beq
\begin{array}{ccc}
{\bf A} & \sigma & \sigma^{-1} \\ \hline
5 & (6,2) & (5,0) \\
4 & (5,1) & (3,2) \\
1 & (2,1) & \textcolor{red}{(4,1)} \\
2 & (3,2) & (7,0) \\
0 & (1,0) & (2,2) \\
6 & (7,0) & \textcolor{red}{(1,1)} \\
3 & (4,0) & (6,0) \\ \hline
\end{array}
\eeq
\QED
\end{proof}
The following lemma 
is a simple extension of 
\cite{DF2}, Proposition 4.1. 
\begin{lemma}
\label{Gessel}
Let 
$\sigma \in G_{p, n}$
with 
$d(\sigma) = d$  
and let 
\beq
c_{ij}^d
:= \sharp  \left\{
(\tau, \mu) \in G_{p, n} \times G_{p, n}
\, | \, 
d(\tau)=i, \; d(\mu) = j, 
\;
\tau \mu = \sigma 
\right\}.
\eeq
Then 
$c_{ij}^d$ 
is independent of the choice of 
$\sigma$ 
such that 
$d(\sigma) = d$
and satisfies the following equation. 
\beq
\sum_{i, j \ge 0}
\frac {
c_{ij}^d s^i t^j
}
{
(1-s)^{n+1} (1-t)^{n+1}
}
=
\sum_{a, b \ge 0}
\left( 
\begin{array}{c}
n+pab + a + b - d 
\\
n
\end{array}
\right)
s^a t^b.
\eeq
\end{lemma}
\begin{proof}
By Lemma \ref{composition}
we have 
\beq
&&
\sum_{\mu}
\left(
\begin{array}{c}
n + a - d (\mu) 
\\
n
\end{array}
\right)
\mu^{-1}
\sum_{\tau}
\left(
\begin{array}{c}
n + b - d(\tau) 
\\
n
\end{array}
\right)
\tau^{-1}
\\
&&
\qquad\qquad
=
\sum_{\sigma}
\left(
\begin{array}{c}
n + pab + a + b - d(\sigma)
\\
n
\end{array}
\right)
\sigma^{-1}.
\eeq
We multiply both sides by 
$s^a \cdot t^b$, 
take summation in 
$a, b \ge 0$, 
and then take the coefficient of 
$\sigma$. 
Noting that 
$\tau \mu = \sigma$, 
we have 
\beq
&&
\sum_{\tau \mu = \sigma}
\sum_a
\left(
\begin{array}{c}
n + a - d(\mu)
\\
n
\end{array}
\right)
s^a
\sum_b
\left(
\begin{array}{c}
n + b - d(\tau)
\\
n
\end{array}
\right)
t^b
\\
&=&
\sum_{\tau \mu = \sigma}
\frac {
s^{d(\mu)}
}
{
(1 - s)^{n+1}
}
\cdot
\frac {
t^{d(\tau)}
}
{
(1-t)^{n+1}
}
\\
&=&
\sum_{i, j}
c_{ij}^d
\frac {
s^{i}
}
{
(1 - s)^{n+1}
}
\cdot
\frac {
t^{j}
}
{
(1-t)^{n+1}
}.
\eeq
\QED
\end{proof}

\noindent
{\it Alternative Proof of Proposition 
\ref{onestep}}\\
By
Lemma \ref{shufflecounting}, 
we have 
\beq
{\bf P}(d (\sigma_r) = j)
&=&
\sum_{i \ge 0}
\sum_{d(\sigma^{-1})=i, \; d(\sigma)=j}
{\bf P}(\sigma_r = \sigma)
\\
&=&
\sum_{i \ge 0}
c_{ij}^0
\left(
\begin{array}{c}
n + \frac {b^r-1}{p} - i 
\\
n
\end{array}
\right)
\cdot
b^{-rn}.
\eeq
Putting 
$d=0$
in 
Lemma \ref{Gessel}, 
and taking the coefficient of 
$s^{m}$, 
$m :=
\frac{b^r-1}{p}$
on both sides, we have 
\beq
&&
\sum_{i, k \ge 0}
c_{ik}^0
\left(
\begin{array}{c}
n + \frac {b^r - 1}{p} - i  \\ n
\end{array}
\right)
\frac {t^k}{(1-t)^{n+1}}
%
=
\sum_d
\left(
\begin{array}{c}
n + \frac {b^r-1}{p} (pd+1) + d \\ n
\end{array}
\right)
t^d.
\eeq
We multiply 
$(1-t)^{n+1}$
both sides and take the coefficient of 
$t^j$
on both sides. 
\beq
&&
{\bf P}(d(\sigma_r) = j)
\\
&=&
b^{-rn}
[t^j]
\left\{
(1-t)^{n+1}
\sum_d
\left(
\begin{array}{c}
n + \frac {b^r-1}{p} (pd+1) + d \\ n
\end{array}
\right)
t^d
\right\}
\\
&=&
b^{-rn}
\sum_r
(-1)^r
\left(
\begin{array}{c}
n+1 \\ r
\end{array}
\right)
\left(
\begin{array}{c}
n + b^r (j-r)
+
\frac {b^r - 1}{p}
\\
n
\end{array}
\right)
\\
&=&
{\bf P}\left(
\kappa_r = j \, | \, \kappa_0 = 0
\right).
\eeq
In the last equality, 
we used Lemma \ref{composition}. 
\QED

\vspace*{1em}

\noindent {\bf Acknowledgement }
The authors
would like to thank Professor Masao Ishikawa 
for valuable discussions.
This work is partially supported by 
JSPS grant Kiban-C no.26400145(F.N.)
and no.26400149(T.S.).

%
\small

\end{document}